\documentclass[12pt, a4paper]{article}

\usepackage{amssymb}
\usepackage{amsmath}
\usepackage{amsthm}

\newtheorem{lemma}{Lemma}[section]
\newtheorem{thm}[lemma]{Theorem}
\newtheorem{prop}[lemma]{Proposition}
\newtheorem{cor}[lemma]{Corollary}

\newtheorem{hyp}[lemma]{Hypothesis}
\newtheorem{definition}[lemma]{Definition}

\begin{document}

\setlength{\parindent}{0mm}

\newcommand{\m}{$\,\textrm{max}\,$}
\newcommand{\w}{\widehat}
\newcommand{\wi}{\widehat}
\newcommand{\ov}{\overline}
\newcommand{\N}{\mathbb{N}}
\def \P{\mathbb{P}}

\newcommand{\E}{\mathcal{E}}
\newcommand{\K}{\mathcal{K}}
\newcommand{\sym}{\mathcal{S}}
\newcommand{\A}{\mathcal{A}}

\newcommand{\Z}{\mathbb{Z}}

\newcommand{\wt}{\widetilde}
\newcommand{\wh}{\widehat}
\newcommand{\ti}{\tilde}

\newcommand{\ch}{$\textrm{char}$}
\newcommand{\sy}{$\,\textrm{Syl}$}
\newcommand{\au}{$\textrm{Aut}$}
\newcommand{\Out}{$\textrm{Out}$}
\newcommand{\PSL}{$\textrm{PSL}$}
\newcommand{\PSU}{$\textrm{PSU}$}
\newcommand{\PGL}{$\textrm{PGL}$}
\newcommand{\PGU}{$\textrm{PGU}$}
\newcommand{\PGO}{$\textrm{PGO}$}
\newcommand{\PSO}{$\textrm{PSO}$}
\newcommand{\PGaL}{P\Gamma L}
\newcommand{\GL}{$\textrm{GL}$}
\newcommand{\GU}{$\textrm{GU}$}
\newcommand{\Sp}{$\textrm{Sp}$}
\newcommand{\PSp}{$\textrm{PSp}$}
\newcommand{\Sz}{$\textrm{Sz}$}
\newcommand{\SL}{$\textrm{SL}$}
\newcommand{\SU}{$\textrm{SU}$}
\newcommand{\F}{$\textrm{GF}$}
\newcommand{\Fi}{$\textrm{Fi}$}
\newcommand{\FO}{\textrm{fix}_{\Omega}}
\newcommand{\FL}{\textrm{fix}_{\Lambda}}
\newcommand{\fixO}{\textrm{fix}_{\Omega}}
\newcommand{\fixL}{\textrm{fix}_{\Lambda}}
\newcommand{\out}{$\textrm{Out}$}
\newcommand{\Sym}{\mathcal{S}}
\newcommand{\Alt}{\mathcal{A}}
\newcommand{\rank}{$\textrm{r}$}

\def \<{\langle }
\def \>{\rangle }
\def \L{\mathcal{L}}

\begin{center}
\Large{\textbf{Transitive permutation groups with trivial four point
stabilizers}}

\vspace{0.2cm} \small{Kay Magaard and Rebecca Waldecker}
\end{center}

\normalsize

\vspace{2cm}

\begin{center}
\textbf{Abstract}
\end{center}

\small
In this paper we analyze the structure of transitive permutation
groups that have trivial four point stabilizers, but some nontrivial
three point stabilizer. In particular we give a complete, detailed
classification when the group is simple or quasisimple. This paper is motivated by questions concerning the relationship between fixed points of automorphisms of Riemann surfaces and Weierstra{\ss} points and is a continuation of the authors' earlier work.

\vspace{1cm}
\normalsize

\section{Introduction}

\vspace{0.2cm} In \cite{MW} we study permutation groups that act
nonregularly such that every nontrivial element has at most two
fixed points. The motivation for the study are questions concerning
automorphisms and Weierstra\ss \ points of Riemann surfaces, and for
these the permutation groups in which nontrivial elements fix at
most three or at most four points are relevant. In the present
article we look at transitive permutation groups where some
nontrivial element fixes three points, but all four point
stabilizers are trivial.

\begin{thm}\label{simple3fp}
Suppose that $G$ acts faithfully and transitively on a set $\Omega$.
Suppose that the four point stabilizers are trivial, but that some
three point stabilizer is nontrivial. If $G$ is simple and $\omega
\in \Omega$, then one of the following holds:
\begin{enumerate}
\item $G_\omega$ is not cyclic and one of the following is true:
\begin{enumerate}
\item
$G \cong \A_5$, $|\Omega|=15$ and $G_\omega \in \sy_2(G)$.
\item
$G \cong \A_6$, $|\Omega| \in \{6, 15\}$ and $G_\omega$ is
isomorphic to $\A_5$ or $\Sym_4$, respectively.
\item
$G \cong \PSL_2(7)$, $|\Omega|=7$ and $G_\omega \cong \Sym_4$.
\item
$G \cong \A_7$, $|\Omega| = 15$ and $G_\omega \cong \PSL_2(7)$.
\item
$G \cong \PSL_2(11)$, $|\Omega|=11$ and $G_\omega \cong \A_5$.
\item
$G \cong M_{11}$, $|\Omega|=11$ and $G_\omega \cong \A_5$.

\end{enumerate}

\item $G_\omega$ is cyclic of order prime to $6$ and one of the following is true:
\begin{enumerate}
\item
$G \cong \PSL_3(q)$ and $|G_\omega| = q^2+q+1/(3,q-1)$.

\item
$G \cong \PSU_3(q)$ and $|G_\omega| = q^2-q+1/(3,q+1)$.

\item
$G \cong \PSL_4(3)$, $|\Omega|=2^7 \cdot 3^6 \cdot 5$ and $|G_\omega| = 13$.

\item
$G \cong \PSU_4(3)$, $|\Omega|=2^7 \cdot 3^6 \cdot 5$ and $|G_\omega| = 7$.

\item
$G \cong \PSL_4(5)$, $|\Omega|=2^7 \cdot 3^2 \cdot 5^6 \cdot 13$ and
$|G_\omega| = 31$.

\item
$G \cong \A_7$, $|\Omega| = 360$ and $|G_\omega| = 7$.

\item
$G \cong \A_8$, $|\Omega|=2880$ and $|G_\omega| = 7$.

\item
$G \cong M_{22}$, $|\Omega|=2^7 \cdot 3^2 \cdot 5 \cdot 11$ and
$|G_\omega| = 7$.

\end{enumerate}
\end{enumerate}
\end{thm}

\begin{thm}\label{almostsimple3fp}
Suppose that $G$ acts faithfully and transitively on a set $\Omega$.
Suppose that the four point stabilizers are trivial, but that some
three point stabilizer is nontrivial. If $G$ is almost simple, but
not simple and if $\omega \in \Omega$, then one of the following
holds:
\begin{enumerate}
\item
There is a prime $p$ such that $G \cong \au(\PSL_2(2^p)) =
\au(\SL_2(2^p))$, and $\Omega$ is the set of $1$-spaces of the
natural module of $\SL_2(2^p)$. (This includes the example where $G
\cong \Sym_5$ in its natural action on five points.)

\item $G \cong \PGL_3(q)$ with $(q-1,3) = 3$, $|\Omega| = q^3(q^2-1)$ and $G_\omega$ is cyclic of order $(q^3-1)/(q-1)$.

\item $G \cong \PGU_3(q)$ with $(q+1,3) = 3$, $|\Omega| = q^3(q^2+1)$ and $G_\omega$ is cyclic of order $(q^3+1)/(q+1)$.

\end{enumerate}
\end{thm}

Almost 40 years ago Pretzel and Schleiermacher \cite{PS1} studied an
important special case of our present situation, namely they
investigated transitive permutation groups in which, for a fixed
prime $p$, every nontrivial element fixes either $p$ or zero points.
They stated that one would like to prove that either $G$ contains a
regular normal subgroup of index $p$ or that $G$ contains a normal
subgroup $F$ of index $p$ such that $F$ acts as a Frobenius group on
its $p$ orbits.

\begin{thm} \label{main}
Suppose that $G$ acts faithfully and transitively on a set $\Omega$.
Suppose that the four point stabilizers are trivial, but that some
three point stabilizer is nontrivial. Then $G$ has order divisible
by $3$ and if $\omega \in \Omega$, then one of the following holds:
\begin{enumerate}
 \item
 $|G_\omega|$ is even and one of the following is true:
\begin{enumerate}
\item
$G$ has a normal $2$-complement.

\item
$G$ has dihedral or semidihedral Sylow $2$-subgroups and $4$ does
not divide $|G_\omega|$. In particular $G_\omega$ has a normal
$2$-complement.

\item
$G_\omega$ contains a Sylow $2$-subgroup $S$ of $G$ and $G$ has a
strongly embedded subgroup.

\item
$|G:G_\omega|$ is even, but not divisible by $4$ and $G$ has a
subgroup of index $2$ that has a strongly embedded subgroup.

\end{enumerate}

\item
$|G_\omega|$ is odd and one of the following is true:
\begin{enumerate}
\item $G$ has a normal subgroup $R$ of order $27$ or $9$, and $G/R$ is
isomorphic to $\Sym_3$, $\Alt_4$, $\Sym_4$, to a fours group or to a
dihedral group of order $8$.
\item $G$ has a regular normal subgroup.
\item $G$ has a normal subgroup $F$ of index $3$ which acts as
a Frobenius group on its three orbits.
\item $G$ has a normal subgroup $N$ which acts
semiregularly on $\Omega$ such that $G/N$ is almost simple
and $G_\omega$ is cyclic.
\end{enumerate}

\end{enumerate}

\end{thm}

This paper is structured as follows. After fixing some standard
notation, we introduce examples which are typical for the situation
that we analyze later on. Then we move on to prove results about the
local structure of the groups under consideration and collect enough
information to bring the Classification of Finite Simple Groups into
action in an efficient way. Sections 3 to 5 deal with particular
classes of simple and quasisimple groups. Then in Section 6 we
collect this information for the proof of Theorems 1.1 and 1.2. Then
we give the proof of Theorem 1.3. after which we show that the
possibilities arising in
Theorem 1.3.2~(b) are like the examples given in Section 2.1.\\

\begin{center}

\textbf{Acknowledgments}

\end{center}

The second author wishes to thank the University of Birmingham for
the hospitality during numerous visits and the Deutsche
Forschungsgemeinschaft for financial support of this project. The
first author wishes to thank the Martin-Luther-Universit\"at
Halle-Wittenberg for its hospitality during his visits. Both authors
wish to thank Chris Parker for bringing the work of Pretzel and
Schleiermacher to their attention.

\section{Preliminaries}

\vspace{0.2cm} In this paper, by ``group'' we always mean a finite
group, and by ``permutation group'' we always mean a group that acts
faithfully.

In this chapter let $\Omega$ denote a finite set and let $G$ be a
permutation group on $\Omega$.

\medskip
\textbf{Notation}

\smallskip
Let $\omega \in \Omega$ and $g \in G$, and moreover let $\Lambda
\subseteq \Omega$ and $H \le G$.

Then
$H_{\omega}:=\{h \in H \mid \omega^h=\omega\}$ denotes the stabilizer of $\omega$ in $H$,

$\FL(H):=\{\omega \in \Lambda \mid \omega^h=\omega$ for all $h \in
H\}$ denotes the fixed point set of $H$ in $\Lambda$ and we write
$\FL(g)$ instead of $\FL(\<g\>)$.

We write $\omega^H$ for the $H$-orbit in $\Omega$ that contains $\omega$.

Whenever $n,m \in \N$, then we denote by $(n,m)$ the greatest common
divisor of $n$ and $m$. Moreover we write $\Z_n$ (or sometimes just $n$) for a cyclic group
of order $n$.

\begin{lemma}\label{charfrob}
Suppose that $G$ has a non-trivial proper subgroup $H$ such that the
following holds: ~Whenever $1 \neq X \le H$, then $N_G(X) \le H$.

Then $G$ is a Frobenius group with Frobenius complement $H$.
\end{lemma}

\begin{proof}
This is Lemma 2.1 in \cite{MW} .
\end{proof}

\begin{lemma}\label{tuedel}
Suppose that $G$ acts transitively on the set $\Omega$ and that
$\alpha \in \Omega$. Let $1 \neq X \le G_\alpha$. Then the following
hold:

\begin{enumerate}
\item[(a)]
If $\alpha$ is the unique fixed point of $X$, then $N_G(X) \le
G_\alpha$.

\item[(b)]
If $X$ has exactly two fixed points, then $N_{G_\alpha}(X)$ has
index at most 2 in $N_G(X)$.

\item[(c)]
If $X$ has exactly three fixed points, then $N_{G_\alpha}(X)$ has
index at most $3$ in $N_G(X)$.
\end{enumerate}

\end{lemma}

\begin{proof}
Assertion (a) holds in any permutation group. As $N_G(X)$ acts on
$\FO(X)$, we see in the case of (b) that $N_G(X)/N_{G_\alpha}(X)$ is
isomorphic to a subgroup of $\Sym_2$. In (c) let $K$ denote the
kernel of the action of $N_G(X)$ on $\FO(X)$. Then $N_G(X)/K$ is
isomorphic to a subgroup of $\Sym_3$. If this factor group is
isomorphic to a proper subgroup of $\Sym_3$, then (c) holds.
Otherwise we note that there is $g \in N_G(X)$ that fixes $\alpha$
and interchanges the other two points in $\FO(X)$. Hence $g \in
G_\alpha$ and $|N_G(X):N_{G_\alpha}(X)|=3$. So again (c) holds.
\end{proof}

\begin{lemma}\label{frob}
Suppose that $G$ is a $\{2,3\}'$-group and that $G$ acts
transitively, nonregularly on a set $\Omega$ such that four point
stabilizers are trivial. Then $G$ is a Frobenius group.
\end{lemma}

\begin{proof}
This follows from Lemmas \ref{tuedel} and \ref{charfrob}.
\end{proof}

\begin{hyp}\label{3fix}
Suppose that $(G, \Omega)$ is such that $G$ acts transitively,
nonregularly on the set $\Omega$, that four point stabilizers are
trivial and that some three point stabilizer is nontrivial.
\end{hyp}

Note that Hypothesis \ref{3fix} implies that $|\Omega| \geq 5$
because nontrivial permutations on four or fewer points can have at
most two fixed points.

\begin{lemma} \label{cases}
If $(G, \Omega)$ satisfies Hypothesis \ref{3fix} and $\omega \in \Omega$, then one of the
following is true:
\begin{enumerate}
\item[(1)]
$|G_\omega|$ is even.
\item[(2)]
$G_\omega$ is a Frobenius group of odd order, where the Frobenius
complements are three point stabilizers.
\item[(3)]
$|\FO(G_\omega)| = 3$ and $|G_\omega|$ is odd.
\end{enumerate}
\end{lemma}

\begin{proof}
We suppose that $|G_\omega|$ is odd and that $|\FO(G_\omega)| \neq
3$. Thus we need to show that the statements in (2) hold, in
particular that $G_\omega$ is a Frobenius group.

Hypothesis \ref{3fix} implies that there exists a set $\Delta$ of
size 3 such that $\omega \in \Delta$ and such that the point-wise
stabilizer $H$ of $\Delta$ in $G$ is nontrivial. Let $1 \neq X \le
H$. Then $X$ acts semiregularly on $\Omega \setminus \Delta$ and
$N_G(X)$ leaves $\Delta$ invariant. Since $|G_\omega|$ is odd and
$|\Delta|=3$, this implies that $N_G(X)$ has odd order. Hence
$|N_G(X):N_H(X)| \in \{1,3\}$ and this holds for all $1 \neq X \leq
H$.

Next we observe that if there is a nontrivial subgroup $X$ of $H$
such that $|N_G(X):N_H(X)| = 3$, then all $g \in N_G(X) \setminus
N_H(X)$ act transitively on $\Delta$; i.e. they fix no point of
$\Delta$.

Thus $|N_{G_\omega}(X):N_H(X)|=1$ for all $1 \neq X \leq H <
G_\omega$ and Lemma \ref{charfrob} implies that $G_\omega$ is a
Frobenius group where $H$ is a Frobenius complement. This is our
claim.
\end{proof}

We recall that a subgroup $H$ of $G$ is t.i. if and only if, for all
$g \in G$, either $H \cap H^g = 1$ or $H^g=H$.

\begin{cor}\label{FrobTo3}
Suppose that $(G, \Omega)$ satisfies Hypothesis \ref{3fix} and that
$|\Omega| \geq 7$. Let $\omega \in \Omega$ and suppose that
$G_\omega$ is a Frobenius group of odd order with a Frobenius
complement $H$ that is a three point stabilizer. Let $\Lambda :=
G/H$ (with the natural action of $G$ by right multiplication). Then
$(G,\Lambda)$ satisfies Hypothesis \ref{3fix}. Moreover if $h \in
G^\#$ stabilizes $\Lambda$, then $|\FL(h)| = 3$.
\end{cor}

\begin{proof}
As $G$ is not a Frobenius group by Hypothesis \ref{3fix} and the
point stabilizers have odd order, Lemma \ref{charfrob} implies that
there exists some $1 \neq X \leq H$ such that $|N_G(H):N_H(X)|=3$.
Now $X$ acts semiregularly on $\Omega \setminus \FO(H)$, and since
$|\Omega| \geq 7$, this implies that the set-wise stabilizer of
$\FO(H)$ is properly larger than $H$. Therefore $|N_G(H):H| = 3$.
Also if $h \in H \cap H^g$ and $H \neq H^g$, then $h$ fixes $\FO(H)
\cup \FO(H^g) \neq \FO(H)$ and hence $h = 1$ by Hypothesis
\ref{3fix}. So $H$ is t.i. and $|N_G(H):H| = 3$, which implies our
claim.
\end{proof}

\subsection{Examples}

Here we describe some series of examples for Hypothesis \ref{3fix}.
In particular we classify all possibilities where $\Omega$ has five
or six elements.

\begin{lemma} \label{kleinersechs}
If $(G,\Omega)$ satisfies Hypothesis \ref{3fix} and $|\Omega| \leq
6$, then one of the following is true:
\begin{enumerate}
 \item[(1)] $|\Omega| = 5$ and $G = \Sym_5$,
 \item[(2)] $|\Omega| = 6$ and $G = \Alt_6$,
 \item[(3)] $|\Omega| = 6$ and $\Alt_3 \wr \Sym_2 \leq G \leq (\Sym_3 \wr \Sym_2) \cap \Alt_6$
 (two possibilities in total).
\end{enumerate}

\end{lemma}
\begin{proof}
Hypothesis \ref{3fix} implies that some element $g \in G$ has three
fixed points on
$\Omega$ and that $|\Omega| \geq 5$.
In the following we view $G$ as a subgroup of $\Sym_6$.\\

If $|\Omega| = 5$, then $g$ is a $2$-cycle. As $5$ is prime, the
hypothesis that $G$ is transitive implies that $G$ is primitive. Now
$G$ is a primitive permutation group on five points that contains a
transposition, so
$G = \Sym_5$ as stated in (1).\\

If $|\Omega| = 6$, then $g$ is a $3$-cycle. Without loss $g=(456)$,
so $g$ lies in the point stabilizer $G_1$. The $2$-cycles from
$\Sym_6$ have four fixed points, therefore Hypothesis \ref{3fix}
implies that $(1,2)^{\Sym_6} \cap G = \varnothing$. If $G$ acts
primitively on $\Omega$, then it follows that $G=\Alt_6$ which leads
to (2). Possibility (2) does in fact occur as an example, as an
inspection of the conjugacy classes shows. If $G$ is not primitive
on $\Omega$, then, since $G$ contains a $3$-cycle, it is a subgroup
of $\Sym_3 \wr \Sym_2$. Now $|G| = 6 \cdot |G_1| \geq 18$ which
implies that $\Alt_3 \wr \Sym_2 \leq G$. On the other hand $G \neq
\Sym_3 \wr \Sym_2$ as $G$ does not contain $2$-cycles. Therefore $G
\leq (\Sym_3 \wr \Sym_2) \cap \Alt_6$ and (3) follows.
\end{proof}

Having considered small examples we also look at sharply
$4$-transitive permutation groups. We note that any element of such
a group that fixes four points is the identity element. Moreover a
three point stabilizer in such a group is transitive on the set of
points that are not fixed, and in particular it is nontrivial if the
size of the set is at least $5$. The next result is by Jordan and
can be found as Theorem 3.3 in Chapter XII of \cite{HB}.

\begin{lemma}\label{4scharftransitiv}
If $G$ is sharply $4$-transitive, then $G$ is one of $\Sym_4,
\Sym_5, \Alt_6, M_{11}$.
\end{lemma}

Thus we see that $\Sym_5, \Alt_6, M_{11}$ in their actions on $5,6$
or $11$ points, respectively, are examples satisfying Hypothesis
\ref{3fix}.

\begin{lemma} \label{maximaleKlasse}
Suppose that $P$ is a $3$-group of order at least $27$ and that $H
\leq P$ is a subgroup of order $3$ such that $|C_P(H)| = 9$. Let
$\Omega$ denote the set of right cosets of $H$ in $P$. Then
$(P,\Omega)$ satisfies Hypothesis \ref{3fix}.
\end{lemma}

\begin{proof} Only the conjugates of elements of $H$ have fixed points on $\Omega$.
If $1 \neq h \in H$, then $|\FO(h)| = |N_P(H):H|$. The outer
automorphism group of $H$ has order coprime to $3$, therefore
$N_P(H) = C_P(H)$ and our hypothesis on $|C_P(H)|$ implies that
$|N_P(H):H| = 3$. This proves our claim.
\end{proof}

We recall that a non-abelian $p$-group $P$ is of maximal class if it
possesses a $p$-element $x$ such that $|C_P(x)|=p^2$. Extraspecial
$p$-groups of order $p^3$ are examples of this. The $2$-groups of
maximal class are dihedral, quaternion or semidihedral, whereas for
$p>2$ there are many other examples (see for example in \cite{Hupp},
III.14). Lemma \ref{maximaleKlasse} implies that $3$-groups of
maximal class all give rise to examples for Hypothesis \ref{3fix}.

The next three examples are variants of those introduced in
\cite{MW}.

\begin{lemma} \label{BeispielKoerper}
Let $p$ be a prime and let $A$ denote the additive group, $M$ the
multiplicative group, and ${\mathcal G}$ the Galois group of a
finite field of order $3^p$. Let $G$ be the semidirect product
$(A:M):{\mathcal G}$ and $G_\omega := M:{\mathcal G}$. Let $\Omega$
denote the set of right cosets of $G_\omega$ in $G$. Then $(G,
\Omega)$ satisfies Hypothesis \ref{3fix}.
\end{lemma}
\begin{proof}
We note first that $A$ is a regular normal subgroup of $G$ in its
action on $\Omega$. Thus if $g \in G_\omega$, then $\FO(g) =
|C_A(g)|$. Our claim follows as $1$ and $3$ are the only possible
values for $|C_A(g)|$.
\end{proof}

\begin{lemma} \label{3malFrobenius}
Let $F$ be a Frobenius group with kernel $K$ and complement $H$ and
let $Z$ be a cyclic group of order $3$. Let $G = Z \times F$ and let
$\Omega$ be the set of right cosets of $H$. Then the pair
$(G,\Omega)$ satisfies Hypothesis \ref{3fix}.
\end{lemma}

\begin{proof}
The subgroup $K$ has three orbits on $\Omega$ which are transitively
permuted by $Z$ and fixed set-wise by elements of $H$. If $h \in H$,
then $h$ fixes exactly one point on each $K$-orbit. Our claim
follows.
\end{proof}

We remark that in this last example the number of fixed points of an
element is either $0$ or $3$.

\begin{lemma} \label{3malFrobeniusVerschraenkt}
Let $p,r$ be primes and let $K$ be a field of order $p^{3r}$. Let
$A$ and $M$ be the additive respectively the multiplicative group of
$K$ and let $H$ be a subgroup of the Galois group of $K$ of order
$3$. Let $\Omega$ be the set of right cosets of $M$ in $G:=(A : M):
H$. Then $(G,\Omega)$ satisfies Hypothesis \ref{3fix}.
\end{lemma}

\begin{proof} We first observe that $A$ has three regular orbits in its
action on $\Omega$ which are permuted transitively by $H$; i.e. $A
\rtimes H$ acts regularly on $H$.  If $1 \neq m \in M$ and $\alpha
\in \FO(m)$, then $\alpha^H \subseteq \FO(m)$ because $H$ normalizes
$M$ and $M$ is cyclic. The claim follows.

\end{proof}

We close this section with a result of Fukushima \cite{Fu3} which generalizes a result of Rickman \cite{R},
and leads to yet another fairly general class of examples.

\begin{lemma}\label{Fukushima}
Let $H$ be a finite group and $\alpha \in \au(H)$ of odd prime
order. If the order of $\alpha$ is coprime to $|H|$ and if
$C_H(\alpha)$ is a $3$-group, then $H$ is solvable and more
specifically $H = O_{3,3'}(H)C_H(\alpha)$. If $G :=H \rtimes \<
\alpha \>$, if moreover $\Omega$ is the set of right cosets of $\<
\alpha \>$ in $G$ and $|C_H(\alpha)| = 3$, then the pair
$(G,\Omega)$ satisfies Hypothesis \ref{3fix}.
\end{lemma}

\begin{proof}
The first statement is the combined content of Theorem 1 and
Proposition 3 of Fukushima \cite{Fu3}, whereas the second is a
corollary of the first.
\end{proof}

\subsection{More general properties following from our hypothesis}

\begin{lemma}\label{center}
Suppose that Hypothesis \ref{3fix} holds. Then $|Z(G)| \in \{1,3\}$.
\end{lemma}

\begin{proof}
Let $\alpha \in \Omega$. As $G$ acts faithfully on $\Omega$, we know
that $Z(G)$ intersects $G_\alpha$ trivially. Let $x \in G_\alpha$ be
an element with exactly three fixed points. Then $Z(G) \le C_G(x)$
and hence Lemma \ref{tuedel}~(c) implies that $|Z(G)| \in \{1,3\}$.
\end{proof}

\begin{lemma}\label{sylow}
Suppose that Hypothesis \ref{3fix} holds and let $\alpha \in
\Omega$. Then the following hold:

\begin{enumerate}
\item[(a)]
If some $2$-element in $G_\alpha$ has exactly three fixed points on
$\Omega$, then $G_\alpha$ contains a Sylow $2$-subgroup of $G$.

\item[(b)]
If some $3$-element in $G_\alpha$ has exactly three fixed points,
then $3$ divides $\Omega$. In particular, in this case, $G_\alpha$
does not contain a Sylow $3$-subgroup of $G$.

\item[(c)]
For all primes $p \ge 5$ that divide $|G_\alpha|$, some Sylow
$p$-subgroup of $G$ is contained in $G_\alpha$.
\end{enumerate}

\end{lemma}

\begin{proof}
Suppose that $x \in G_\alpha$ is a $2$-element with exactly three
fixed points. As $x$ has orbits of $2$-power lengths on the set of
points that are not fixed, it follows that $|\Omega|$ is odd.
Therefore $|G:G_\alpha|$ is odd and $G_\alpha$ contains a Sylow
$2$-subgroup of $G$.

For (b) suppose that $y \in G_\alpha$ is a $3$-element with exactly
three fixed points on $\Omega$. The remaining orbits of $y$ on
$\Omega$ have $3$-power lengths and therefore $\Omega$ is divisible
by $3$. This means that $|G:G_\alpha|$ is divisible by $3$ and in
particular $G_\alpha$ does not contain a Sylow $3$-subgroup of $G$.

Finally suppose that $p \in \pi(G_\alpha)$ is such that $p \ge 5$.
Let $x \in G_\alpha$ be an element of order $p$ and let $x \in P \in
\sy_p(G)$. Then Lemma \ref{tuedel}~(b) implies first that $Z(P) \le
G_\alpha$ and then that $P \le G_\alpha$. This finishes the proof.
\end{proof}

\begin{lemma}\label{2cycles}
Suppose that Hypothesis \ref{3fix} holds and that $N \unlhd G$ is
such that all $N$-orbits on $\Omega$ have size $2$. Let
$\ti{\Omega}$ denote the set of $N$-orbits of $\Omega$ and let $K$
denote the kernel of the action of $G$ on $\ti{\Omega}$.

Then $|\Omega| \le 6$ and $(G, \Omega)$ is as in Lemma
\ref{kleinersechs}.
\end{lemma}

\begin{proof}
By hypothesis $N$ is a $2$-group. If $N$ has order $2$, then $N \le
Z(G)$ and this contradicts Lemma \ref{center}. Hence $N$ has order
at least $4$. We set $m:=|\ti{\Omega}|$ and we simplify notation by
denoting the elements of $\Omega$ by $1,...,2m$ and by expressing
elements of $G$ as elements from $\sym_{2m}$. We write
$\ti{\Omega}=\{\{1,2\},...,\{2m-1,2m\}\}$.

Now it is sufficient to prove that $m \le 3$. Hence we assume
otherwise. Our fixed point hypothesis tells us that all elements
from $N^\#$ are a product of at least $m-1$ disjoint transpositions.
Suppose that $t \in N^\#$ induces $(1,2) \cdots (2m-1,2m)$ on
$\Omega$ and let $s \in N^\#$ be such that $s \neq t$. On each
element of $\ti{\Omega}$, only one nontrivial action of $s$ is
possible, namely the action of the corresponding transposition. If
$t$ and $s$ both induce a transposition on $\{1,2\}$, then $s \cdot
t$ fixes $1$ and $2$. Otherwise $s$ fixes $1$ and $2$ and we have
the same two possibilities on $\{3,4\}$. As $|\Omega|$ is even, all
elements from $N^\#$ can only have zero or two fixed points, so
looking at the remaining elements of $\ti{\Omega}$ yields that $s$
or $s \cdot t$ fixes at least four points on $\Omega$. This is
impossible. A similar argument applies if we choose $t$ to already
have two fixed points on $\Omega$. Hence $m \le 3$ as stated.
\end{proof}

\begin{lemma}\label{syl2}
Suppose that Hypothesis \ref{3fix} holds. Let $S \in \sy_2(G)$ and
$\alpha \in \Omega$. Then one of the following holds:

\begin{enumerate}
\item[(1)]
$G_\alpha$ has odd order.

\item[(2)]
$S$ is dihedral or semidihedral and $|S_\alpha|=2$. In particular
$G_\alpha$ has a normal $2$-complement.

\item[(3)]
$|S| \ge 4$, there is a unique $S$-orbit on $\Omega$ of length $2$,
and all other $S$-orbits have length $|S|$. Then $O_2(G)=1$ or
$O_2(G)$ is a fours group and $|\Omega| \le 6$.

\item[(4)]
$|\Omega|$ is odd.

\end{enumerate}
\end{lemma}

\begin{proof}
Suppose that (1) does not hold. Then with Sylow's Theorem we may
suppose that $S_\alpha \neq 1$.

Let $\Delta:=\alpha^S$ and let $n,m \in \N_0$ be such that
$|S_\alpha|=2^n$ and $|S:S_\alpha|=2^m$. First suppose that $m \ge
2$. Let $d$ denote the number of fixed points of $S_\alpha$ on
$\Delta$ and choose $a \in \N_0$ such that $|\Delta|=d+a \cdot 2^n$.
As $n \ge 1$ and $|\Delta|=2^m \ge 4$, we see that $d=2$ and hence
$2^m=2 \cdot (1+a \cdot 2^{n-1})$. This implies that $n=1$ and that
$a= 2^{m-1} - 1$ and Lemma \ref{tuedel}~(b) forces $|C_S(S_\alpha)|
\leq 4$. Thus either $S$ is of order $2$ or of maximal class. For
(2) we assume that $S$ is quaternion. Then $|S| \ge 8$ and
$|S_\alpha|=2$, in particular $G_\alpha$ contains the unique
involution in $S$. But then Lemma \ref{tuedel} forces a subgroup of
index $2$ of $S$ to be contained in $G_\alpha$, which is impossible.
Now 11.9 in \cite{Hupp} yields that $S$ is dihedral or semidihedral.
Moreover $S_\alpha$ has order $2$ which means that $G_\alpha$ has
cyclic Sylow $2$-subgroups and hence a normal $2$-complement. This
is (2).

Now we suppose that $m \le 1$. Then (4) holds or $S_\alpha$ has
index exactly $2$ in $S$. We look at the second case more closely.
By Lemma \ref{tuedel} we know that there exists $\beta \in \Omega$
such that $\alpha \neq \beta$, $S_\alpha=S_\beta$ and some element
in $S$ interchanges $\alpha$ and $\beta$. (In fact all elements in
$S \backslash S_\alpha$ interchange $\alpha$ and $\beta$.) As
$S_\alpha$ already has two fixed points and $|\Omega|$ is even in
this case, it follows that $S_\alpha$ has exactly two fixed points
and hence it has regular orbits on the remaining points of $\Omega$.
It follows that $\Delta:=\{\alpha, \beta\}$ is the unique $S$-orbit
of length $2$ and all other orbits have length $|S|$.

As $|\Omega| > 2$ by hypothesis, there exists a regular $S$-orbit of
$\Omega$ and this means that we may choose $g \in G$ such that
$\Delta  \cap \Delta^g=\varnothing$. Then $D:=S \cap S^g$ stabilizes
the set $\Delta \cup \Delta^g$ of size $4$. Moreover $D$ acts
faithfully on this set by Hypothesis \ref{3fix} and it fixes the
subsets $\Delta$ and $\Delta^g$. Thus $|D| \le 4$ and in particular
$O_2(G)$ has order at most $4$. The point stabilizers have index $2$
in $O_2(G)$ and hence $O_2(G)$ has orbits on $\Omega$ of length $2$.
Now Lemmas \ref{2cycles} and \ref{center} imply all the remaining
details of (3).
\end{proof}

\begin{lemma}\label{2isttoll}
Suppose that Hypothesis \ref{3fix} holds. Let $\alpha \in \Omega$
and suppose further that $G$ is simple. Then one of the following
holds:

\begin{enumerate}
\item[(1)]
$G_\alpha$ has odd order.

\item[(2)]
$G_\alpha$ contains a Sylow $2$-subgroup of $G$. In particular
$G_\alpha$ contains an involution from every conjugacy class.

\item[(3)]
$G$ has dihedral or semidihedral Sylow $2$-subgroups, in particular
$G$ is isomorphic to $\Alt_7$ or $M_{11}$ or there exists an odd
prime power $q$ such that $G \cong \PSL_2(q),~ \PSU_3(q)$ or
$\PSL_3(q)$.
\end{enumerate}
\end{lemma}

\begin{proof}
We go through the cases in Lemma \ref{syl2} with the special
hypothesis that $G$ is simple. The cases (1) and (4) from Lemma
\ref{syl2} give exactly the conclusions (1) and (2) here. If (2)
from the lemma holds, then we use the classification of the
corresponding groups by Gorenstein-Walter and
Alperin-Brauer-Gorenstein, respectively (see \cite{GW} and
\cite{ABG}). This gives the possibilities in (3), so it is only left
to prove that Case (3) of Lemma \ref{syl2} does not occur in a
simple group.

Assume otherwise and let $S \in \sy_2(G)$ be such that $S_\alpha
\neq 1$ and $S$ has order at least $4$. Moreover we assume that $S$
has a unique orbit of length $2$ on $\Omega$ and all other orbits
have length $|S|$. We choose $\beta \in \Omega$ such that $\{\alpha,
\beta\}$ is the $S$-orbit of length $2$, in particular
$S_\alpha=S_\beta$ has index $2$ in $S$.

Let $t \in S \setminus S_\alpha$. Then $t$ interchanges $\alpha$ and
$\beta$ and it fixes all orbits of length $|S|$. As $S$ is not
cyclic by Burnside's Theorem (recall that $G$ is simple), it follows
that $t$ acts as an even permutation on each $S$-orbit and hence on
$\Omega \setminus\{\alpha, \beta\}$. Thus $t$ acts as an odd
permutation on $\Omega$. This means that $G$ possesses a normal
subgroup of index $2$. But $|G| \ge 4$ and $G$ is simple, so this is
impossible.
\end{proof}

\begin{lemma}\label{odd-even}
Suppose that Hypothesis \ref{3fix} is satisfied and that $|\Omega|$
is odd. Then one of the following holds:
\begin{enumerate}

\item[(1)]
$G$ has odd order and $3 \in \pi(G)$.

\item[(2)]
$G$ has a strongly embedded subgroup.

\item[(3)]
$G$ has a normal $2$-complement. In particular $G$ is solvable.

\item[(4)]
$G$ has a normal subgroup $G_0$ of index $2$ that has a strongly
embedded subgroup.
\end{enumerate}

In particular, if $G$ is simple, then $G$ is isomorphic to $\Alt_7$,
to $M_{11}$ or there exists a prime power $q$ such that $G$ is
isomorphic to $\PSL_2(q)$, to $\Sz(q)$, to $\PSU_3(q)$ or to
$\PSL_3(q)$ (with $q$ even).

\end{lemma}

\begin{proof}
Let $\alpha \in \Omega$. Then the transitivity of $G$ on $\Omega$
yields that $|\Omega|=|\alpha^G|=|G:G_{\alpha}|$ and hence
$|G|=|\Omega| \cdot |G_{\alpha}|$. In particular $G_\alpha$ contains
a Sylow $2$-subgroup of $G$.

Suppose that $G_{\alpha}$ has odd order. Then $G$ has odd order, but
it is not a Frobenius group and therefore Lemma \ref{frob} forces $3
\in \pi(G)$. This is (1).

Next suppose that $G_{\alpha}$ has even order and let $S \in
\sy_2(G)$ be contained in $G_\alpha$. We look at the orbits of $S$
on $\Gamma:=\Omega \setminus \{\alpha\}$. As $|\Omega|$ is odd,
there are three possibilities: $S$ fixes two points on $\Gamma$ or
every element in $S^\#$ is fixed point free on $\Gamma$ or $S$ has a
unique orbit of length $2$ on $\Gamma$. Suppose that every element
of $S^\#$ fixes only $\alpha$ and let $H:=G_{\alpha}$. Let $g \in
G\backslash H$ and suppose that $x \in H \cap H^g$ is a $2$-element.
Then $x$ has at least two fixed points on $\Omega$, namely $\alpha$
and $\alpha^g$, and in the present case this forces $x=1$ (because
without loss $x \in S$). It follows that $H \cap H^g$ has odd order
and hence $H$ is a strongly embedded subgroup of $G$. This is (2).

Next suppose that $S$ fixes three points. Let $\Delta$ denote this
fixed point set. Let $M_0$ denote the point-wise stabilizer of
$\Delta$ and let $M:=N_G(M_0)$. We show that $N_G(M)$ is strongly
embedded:

First Lemma \ref{tuedel} and the fact that $S \le M_0$ yield that
$M$ has index at most $3$ in $N_G(M)$. Moreover $|\Omega|$ is odd,
so in particular $M$ does not have two orbits of length $3$ on
$\Omega$, but it has a unique orbit of length $3$ on $\Omega$.
(Otherwise $S$ has too many fixed points.) Therefore $N_G(M)$
stabilizes $\Delta$ and is hence contained in $M$.

Next we let $g \in G \setminus M$ and we choose a $2$-element $t \in
M \cap M^g$. Without loss $t \in S$. Then $t$ stabilizes $\Delta$
and $\Delta^g$. These sets have size $3$ and therefore $t$ has a
fixed point on both of them, moreover it fixes $\Delta$ point-wise.
But also $x \in S^g$ and therefore $x$ fixes $\Delta^g$ point-wise.
The previous paragraph showed that $\Delta \neq \Delta^g$, therefore
$t$ fixes at least four points and this forces $t=1$. Now we have
that $M=N_G(M)$ is strongly embedded in $G$.

The last case is that $S$ has a unique orbit $\{\beta, \gamma\}$ of
length $2$ on $\Gamma$. Then a subgroup of index $2$ of $S$ fixes
three points and therefore the orbit lengths of $S$ on $\Omega$ are
$1$, $2$ and $|S|$.

If $S$ is cyclic, then by Burnside's Theorem (3) holds. So we
suppose that $S$ is not cyclic. Then, in the action on $\Omega
\setminus \{\alpha, \beta, \gamma\}$, the elements of $S$ are even
permutations. Thus the elements of $S^\#$ are odd permutations in
their action on $\Gamma$ (and on $\Omega$), which means that $G$ has
a subgroup $G_0$ of index $2$. Let $S_0:=S \cap G_0$. If $\Omega
\neq \alpha^{G_0}$, then $G_0$ makes two orbits on $\Omega$ which
are interchanged by an element in $N_G(S_0) \setminus G_0$. But $S
\cap G_0$ has different numbers of fixed points on these orbits,
which is impossible. Thus $\Omega = \alpha^{G_0}$ and so
$(G_0,\Omega)$ satisfies Hypothesis \ref{3fix}. Moreover $S_0$ fixes
three points of $\Omega$ and we already showed that this implies
that $G_0$ has a strongly embedded subgroup.

If $G$ is simple, then $G$ is nonabelian because of its nonregular
action on $\Omega$ and hence only case (2) is possible. Then the
main result in \cite{Ben} leads to the groups listed.
\end{proof}

\begin{lemma}\label{syl3}
Suppose that Hypothesis \ref{3fix} holds and that $P \in \sy_3(G)$.
Let $\alpha \in \Omega$. Then one of the following holds:

\begin{enumerate}
\item[(1)]
$G_\alpha$ is a $3'$-group.

\item[(2)]
$P$ is of maximal class, $|P_\alpha| = 3$ and $P_\alpha$ fixes three
points.

\item[(3)]
$|P:P_\alpha|=3$, $P$ has order at least $9$ and $P$ has exactly one
orbit of size $3$ on $\Omega$, all remaining orbits have size $|P|$.
Moreover, in this case, $O_3(G)$ is elementary abelian of order at
most $9$.

\item[(4)]
$3$ does not divide $|\Omega|$.
\end{enumerate}
\end{lemma}

\begin{proof}
Suppose that (1) does not hold. Then $3 \in \pi(G_\alpha)$ and so we
may suppose that $P_\alpha \neq 1$.

Set $\Delta = \alpha^P$ and let $n \in \N$ be such that $|\Delta| =
3^n$. Let $m \in \N$ be such that $|P_\alpha| = 3^m$. First we
suppose that $n \geq 2$. We set $d := |\FO(P_\alpha)|$ and note that
$3^n = d + a3^m$ for some integer $a$. The fact that $\alpha \in
\FO(P_\alpha)$ together with Hypothesis \ref{3fix} implies that $1
\leq d \leq 3$. Thus $d =3$ as $P$ is a $3$-group, and this means
that $P_\alpha$ fixes three points of $\Omega$. We obtain that $3^n
= 3 + a3^m $ and thus $n=2$, $m=1$ and $a = 3^{n-1}-1$.

It follows that $|P_\alpha| = 3$ and now Lemma \ref{tuedel} implies
that $|N_P(P_\alpha)| \leq 9$. As stated after Lemma
\ref{maximaleKlasse}, we now have that $P$ has maximal class. So we
proved (2) in this case.

Next suppose that $n \le 1$. Then $|P:P_\alpha| \le 3$ which means
that $G_\alpha$ contains a subgroup of index at most $3$ of $P$.
Therefore (3) or (4) holds, and it is left to analyze Case (3) more
closely. First we notice that $|P| \ge 9$ and $P_\alpha$ fixes three
points, so $P$ has one orbit $\Delta$ of size $3$ (consisting of
these three points) and all other orbits are of size $|P|$. As
$|\Omega| \ge 5$, there exists a regular $P$-orbit and hence we may
choose $g \in G$ such that $\Delta \cap \Delta^g = \varnothing$.
Then $D:=P \cap P^g$ stabilizes the set $\Delta \cup \Delta^g$ of
size $6$ and it acts faithfully on it by Hypothesis \ref{3fix}. It
follows that $D$ is isomorphic to a subgroup of $\sym_6$ and hence
it is elementary abelian of order at most $9$. As $O_3(G) \le D$,
all statements in (3) are proved.
\end{proof}

\begin{lemma}\label{comp2}
Suppose that Hypothesis \ref{3fix} holds and let $\alpha \in
\Omega$. If $E(G) \neq 1$, then $E(G) \cap G_\alpha \neq 1$.
\end{lemma}

\begin{proof}
Assume that $E(G) \neq 1$, but $E(G) \cap G_\alpha=1$ and let $E$ be
a component of $G$. Let $x \in G_\alpha$ be of prime order $p$.
First we show that $x$ normalizes $E$:

Assume otherwise and let $E_1,...,E_p$ denote the $x$-conjugates of
$E$, where $E=E_1$. Then $L:=E_1 \cdots E_p$ is an $x$-invariant
product of components. Let $e \in E$. Then $e \cdots e^{x^{p-1}} \in
C_L(x)$. By Lemma \ref{tuedel} a subgroup of index $2$ or $3$ of
$C_L(x)$ is contained in $G_\alpha$ and so, by assumption, we see
that $e$ has order $2$ or $3$. It follows that $E$ is a
$\{2,3\}$-group. But this is a contradiction because $E$ is not
solvable.

Thus $x$ normalizes $E$ and Lemma \ref{tuedel} yields that
$G_\alpha$ contains a subgroup of index $2$ or $3$ of $C_E(x)$. By
assumption (and as $E$ is not nilpotent) $C_E(x)$ has order $2$ or
$3$. If the order is $2$, then \cite{Fu2} implies that $E$ is
solvable, which is a contradiction. Hence $|C_E(x)|=3$. If $o(x)
\neq 3$, then the main result in \cite{R} yields that $E$ is
solvable again, which is impossible.

We deduce that $o(x)=3$ and now Theorem 2 in \cite{CC} yields that
$E$ is solvable, which is again a contradiction.
\end{proof}

\begin{lemma}\label{onecomp}
Suppose that Hypothesis \ref{3fix} holds and that $E(G) \neq 1$.
Then $G$ has a unique component.
\end{lemma}

\begin{proof}
We assume otherwise. Let $E$ denote a component of $G$ and let $L$
be a product of components such that $E(G)=E \cdot L$. With Lemma
\ref{comp2} we let $\alpha \in \Omega$ and $1 \neq e \in E(G) \cap
G_\alpha$. Let $a \in E$ and $b \in L$ be such that $e=a \cdot b$.
Lemma \ref{tuedel} implies that a subgroup of index at most $3$ of
$C_E(e)=C_E(a)$ and of $C_L(e)=C_L(b)$, respectively, lies in
$G_\alpha$. Moreover $G_\alpha$ does not contain any normal subgroup
of $G$ and hence $G_\alpha$ does not contain a component, again with
Lemma \ref{tuedel}. As $G$ has more than one component by
assumption, it follows that all components intersect $G_\alpha$
trivially. In particular $a,b \notin G_\alpha$ and the groups
$C_E(a)$ and $C_L(b)$ have order $2$ or $3$. The first case is
impossible by Burnside's Theorem, and in the second case the main
result in \cite{FT2} forces $E \cong L \cong \PSL_2(7)$ and in
particular $G_\alpha \cap E(G)=\langle e \rangle$ and $e$ fixes
three points of $\Omega$. From the structure of $\PSL_2(7)$, there
is an involution $t \in EL$ that inverts $e$ and hence fixes one of
the three fixed points of $e$. Let $\gamma$ denote this fixed point
and let $g \in G$ be such that $\alpha^g=\gamma$. Then $G_\gamma
\cap E(G)$ contains elements of order $3$ and $2$, which is
impossible.
\end{proof}

\begin{lemma}\label{hypcomp}
Suppose that Hypothesis \ref{3fix} holds and that $E$ is a component of $G$.
Then one of the following holds:

\begin{enumerate}
\item[(a)]
There exists a power $q$ of $2$ such that $E \cong \PSL_2(q)$,
$|G:E|$ is prime and every element from $G \setminus E$ induces a
field automorphism on $E$. For all $\alpha \in \Omega$, we have that
$|G_\alpha|=q \cdot (q-1) \cdot |G:E|$ and moreover $E_\alpha$ does
not contain any elements that fix three points.

This includes the special case where $E \cong \A_5$ and $G \cong \sym_5$.

\item[(b)]
There exists $\alpha \in \Omega$ such that $(E,\alpha^E)$ satisfies
Hypothesis \ref{3fix}.
\end{enumerate}

\end{lemma}

\begin{proof}
As $E(G) \neq 1$ by hypothesis, we know from Lemmas \ref{comp2} and
\ref{onecomp} that $E$ is the unique component of $G$ and that $E$
intersects the points stabilizers nontrivially. Hence let $\alpha
\in \Omega$ and $\Delta:=\alpha^E$. Then $E_\alpha \neq 1$ and hence
$E$ acts transitively and nonregularly on $\Delta$. Moreover $E$
acts faithfully because $E \unlhd G$. As $E$ is a component and thus
not solvable, we know that $|\Delta| \ge 5$ and therefore $(E,
\Delta)$ satisfies Hypothesis \ref{3fix}.

Suppose that $E$ does not have any element that fixes three points
on $\Delta$. Then $(E, \Delta)$ satisfies Hypothesis 1.1 from
\cite{MW} and in particular $Z(E)=1$ by Lemma 2.8 in \cite{MW} and
Lemma \ref{center}. Thus $E$ is simple and Theorem 1.2 in \cite{MW}
applies. We refer to Theorem 5.6 in the same paper for details on
the possible action of $E$ on $\Delta$. We also note that Lemmas
\ref{tuedel}~(a) and (b) and
\ref{center} force $F(G)=1$.\\

\textbf{Case 1:} $E \cong \A_5$.

We know that $E = F^*(G)$ and hence $G$ acts faithfully on $E$.
As $(E, \Delta)$ does not satisfy Hypothesis \ref{3fix}, but
$(G, \Omega)$ does, it follows that $G \neq E$ and hence
$G \cong \sym_5$ as stated.\\

\textbf{Case 2:} $E \cong \PSL_3(4)$.

Here the only possibility for the action is that $E_\alpha$ has order $5$. In particular $E_\alpha$ is a Sylow subgroup of $E$.
A Frattini argument yields that $G=EN_G(E_\alpha)$. As $G \neq E$ and $|N_E(E_\alpha)|=10$, Lemma \ref{tuedel} implies that some $g \in G \setminus E$ is contained in $G_\alpha$.
Therefore $2$ or $3$ is contained in $\pi(G_\alpha)$.
If $2 \in \pi(G_\alpha)$, then by Lemma \ref{syl2} an index $2$ subgroup of a Sylow $2$-subgroup of $G$ is contained in $G_\alpha$. But this is impossible because $E_\alpha$ has odd order.
If $3 \in \pi(G_\alpha)$, then also $2 \in \pi(G_\alpha)$ by Lemma \ref{tuedel}. (For information about $\au(\PSL_3(4))$ see for example \cite{ATLAS}.)
We already excluded this.\\

\textbf{Case 3:} $E \cong \PSL_2(7)$.

We recall that $E_\alpha \cong \Alt_4$. The point stabilizers in
$\PGL_2(7)$ grow by a factor of either $2$ or $1$. Inspection of the
maximal subgroups of $\PGL_2(7)$ shows that the former case does not
happen. In the latter case the centralizer order of the inner
involution grows by a factor of $2$ while the order of the point
stabilizer does not. This implies that the number of fixed points of
the involution on $\Omega$ is $4$, and this violates Hypothesis
\ref{3fix}.\\

\textbf{Case 4:} There exists a prime power $q$ such that $E \cong
\PSL_2(q)$.

Using Hypothesis \ref{3fix} choose $x \in G_\alpha$ such that $x$ fixes three points on $\Omega$ and induces an outer automorphism on $E$. Lemma \ref{tuedel} implies that
a subgroup of index $3$ of $C_E(x)$ is contained in $E_\alpha$.

First suppose that $x$ induces a field automorphism. Then it follows from the possible structure of point stabilizers that $C_E(x)$ is a solvable subfield subgroup and we see that $2 \in \pi(E_\alpha)$. Moreover $q$ is a power of $2$ or of $3$.
If $q$ is odd, then $E_\alpha$ contains a fours group from $C_E(x)$ and this is impossible.
If $q$ is even, then $E_\alpha$ has order $q \cdot (q-1)$. Moreover $x$ induces an automorphism of prime order. Hence (b) holds in this case.

Next suppose that $x$ induces a diagonal automorphism.
Then $G_\alpha$ contains an involution that fixes three points, and hence Lemma \ref{sylow}~(a) forces $G_\alpha$ to contain a Sylow $2$-subgroup of $G$, and in particular of $E$.
This is impossible because $E_\alpha$ does not contain a Sylow $2$-subgroup of $E$.\\

\textbf{Case 5:} There exists a prime power $q$ such that $E \cong
\Sz(q)$.

Let $x \in G \setminus E$ be such that $x \in G_\alpha$ and $x$
fixes three points on $\Omega$. Then $x$ induces a field
automorphism on $E$ and hence $C_E(x)$ is a subfield subgroup. Now
any subfield group contains $\Sz(2)$, a group of order $20$, and
then Lemma \ref{tuedel} implies that $E_\alpha$ has an element of
order $5$. But we know from \cite{MW} that  $|E_\alpha|_{2'}
=(q-1)$. Since $(q-1)$ is not divisible by $5$ (because $q$ is a
power of $2$ with odd exponent), we
see that $E$ cannot be a Suzuki group. \\

These are all possible cases by \cite{MW}, hence the proof is complete.
\end{proof}

\begin{thm}\label{minimalnormal}
Suppose that Hypothesis \ref{3fix} holds and that $N$ is a minimal normal subgroup of $G$. Let $\alpha \in \Omega$. Then one of the following holds:

\begin{enumerate}

\item[(a)]
All Sylow subgroups of $G_\alpha$ have rank 1.

\item[(b)]
$N$ is a $2$-group. Moreover $N$ is a fours group whose involutions
act without fixed points on $\Omega$ or $|N:N_\alpha|=2$ and
$N_\alpha$ fixes two points.

\item[(c)]
$N$ is a $3$-group. Moreover $G$ has Sylow $3$-subgroups of maximal
class or $|N:N_\alpha|=3$, $N_\alpha$ fixes three points and $|N|
\le 9$.

\item[(d)]
$N=E(G)$ and either $N \cong \A_5$ or there exists a $2$-power $q$
such that $N \cong \PSL_2(q)$ or $(N,\alpha^N)$ satisfies Hypothesis \ref{3fix}.
\end{enumerate}
\end{thm}

\begin{proof}
The faithful action of $G$ on $\Omega$ yields that $N \nleq G_\alpha$.
We begin with the case where $N$ is elementary abelian. Let $r$ be a prime such that $N$ is an $r$-group and suppose that (a) does not hold.
Let $p \in \pi(G_\alpha)$ and suppose that $G_\alpha$ contains an elementary abelian subgroup $X$ of order $p^2$.

If $r \ge 5$, then Lemma \ref{sylow}~(c) yields that $r \notin \pi(G_\alpha)$ and hence
the coprime action of $X$ on $N$ yields that
$N=\langle C_N(x) \mid x \in X^\#\rangle$. It follows with Lemma \ref{tuedel} that
$N \le G_\alpha$. This is a contradiction. Therefore $r \in \{2,3\}$.

First suppose that $r=2$. Then $|N| \ge 4$ by Lemma \ref{center}. If
$p=2$, then $2 \in \pi(G_\alpha)$. If $p$ is odd, then $N=\langle
C_N(x) \mid x \in X^\#\rangle$ by coprime action and so, applying
Lemma \ref{tuedel}, it follows again that $2 \in \pi(G_\alpha)$. Let
$S \in \sy_2(G)$ and suppose that $|N:N_\alpha| \neq 2$. Then Lemma
\ref{syl2} yields that $S$ is dihedral or semidihedral. As $G$ has
no normal subgroup of order $2$, we see that $N$ is not cyclic, so
it follows that $N$ is a fours group, that the involutions in $N$
act without fixed points on $\Omega$ and that $G/C_G(N)$ is
isomorphic to $\Sym_3$. This is one of the cases in (b). Otherwise
$|N:N_\alpha|=2$ and we let $t \in N$ be such that $t \notin
G_\alpha$. As $t$ normalizes $N_\alpha$, but does not fix $\alpha$,
there must be a second point $\beta \in \Omega$ that is fixed by
$N_\alpha$ and such that $t$ interchanges $\alpha$ and $\beta$. This is the other case in (b).\\

Next suppose that $r=3$. If $|N|=3$, then the second case in (c)
holds. So we suppose that $|N| \ge 9$ and we argue as in the
previous paragraph. If $p=3$, then $3 \in G_\alpha$, and if $p \neq
3$, then $3 \in G_\alpha$ by coprime action and Lemma \ref{tuedel}.
Now Lemma \ref{syl3} yields the possibilities in (c). We note that,
if $|N:N_\alpha|=3$ and $y \in N$ is such that $y \notin N_\alpha$,
then $N_\alpha$ must fix three points and $y$ interchanges these
three points in a $3$-cycle.

This concludes the case where $N$ is solvable.\\

Next suppose that $N$ is a product of components. Then $E(G) \neq 1$ and hence Lemmas
\ref{onecomp} and \ref{comp2} yield that $N$ is the unique component of $G$.
Then (d) holds by Lemma \ref{hypcomp}.
\end{proof}

When we study simple groups satisfying Hypothesis \ref{3fix} (using
the Classification of Finite Simple Groups), we adapt some of
Aschbacher's notation from Section 9 of \cite{A}. We introduce it
here and use it throughout the following sections.

\begin{definition}
Suppose that $p,q \in \pi(G)$ are prime numbers and let $H \le G$ be
a point stabilizer in $G$.

\begin{itemize}
\item
We write $p \vdash q$ if and only if one of the following holds:

-- $q \ge 5$ and there exists a nontrivial $p$-subgroup $X \le G$
such that $q \in \pi(N_G(X))$.

-- $q=2$ and and there exists a nontrivial $p$-subgroup $X \le G$
such that $4$ divides $|N_G(X)|$.

-- $q=3$ and and there exists a nontrivial $p$-subgroup $X \le G$
such that $9$ divides $|N_G(X)|$.

\item
We write $\rightarrow$ for the transitive extension of $\vdash$.
\end{itemize}
\end{definition}

\begin{lemma}\label{key}
Suppose that Hypothesis \ref{3fix} holds and that $H \le G$ is a
point stabilizer. Suppose further that $q \in \pi(G)$ and $p \in
\pi(H)$. If $p \ge 5$ and $p \rightarrow q$, then $q \in \pi(H)$.
\end{lemma}

\begin{proof}
By definition of $\rightarrow$ it suffices to consider the case
where $p \vdash q$. Lemma \ref{sylow}~(c) gives that $H$ contains a
Sylow $p$-subgroup of $G$. Then by Sylow's Theorem there exists a
nontrivial $p$-subgroup $X$ of $H$ such that $q$ (or $4$ or $9$)
divides $|N_G(X)|$ and therefore Lemma \ref{tuedel} yields that $q
\in \pi(H)$.
\end{proof}

\section{Alternating Groups}

In this chapter we discuss what alternating or symmetric groups
appear as examples for Hypothesis \ref{3fix} and if so, then with
what actions. We begin with some small cases and then bring Lemma
\ref{2isttoll} into play. We use the notation that has been
introduced at the end of the previous section.

\begin{lemma}\label{S4}
Suppose that $G$ is isomorphic to $\Alt_4$ or to $\Sym_4$.
Then there is no set $\Omega$ such that
$(G, \Omega)$ satisfies Hypothesis \ref{3fix}.
\end{lemma}

\begin{proof}
Assume otherwise and let $\alpha \in \Omega$ and $x \in G_\alpha^\#$
be such that $|\FO(x)|=3$. If $x$ is a $2$-element, then Lemma
\ref{tuedel}~(a) yields that $|\Omega|$ is odd and hence
$|\Omega|=3$. This is too small for Hypothesis \ref{3fix}. If $x$ is
a $3$-element, then $G_\alpha$ contains a Sylow $3$-subgroup of $G$
(because this has only order $3$) and this contradicts Lemma
\ref{sylow}~(b).
\end{proof}

\begin{lemma}\label{S5}
Suppose that $G$ is isomorphic to $\Alt_5$ or $\Sym_5$ and that $\Omega$ is a set such that $(G, \Omega)$
satisfies Hypothesis \ref{3fix}. Then $|\Omega|=15$ and the action of $G$ is as on the set of cosets of a
Sylow $2$-subgroup, or $G \cong \Sym_5$, $|\Omega|=5$ and $G$ acts naturally.
\end{lemma}

\begin{proof}
The action of $G$ on the set of cosets of a Sylow $2$-subgroup satisfies Hypothesis \ref{3fix}, as does
the natural action of $\Sym_5$ on a set with five elements, so we need to show that these are the only
possibilities.
Suppose that $(G, \Omega)$ satisfies Hypothesis \ref{3fix} and let $\alpha \in G_\alpha$ and $x \in G_\alpha^\#$
be such that $|\FO(x)|=3$.

Assume that $x$ is a $5$-element. The nontrivial orbits of $x$ have lengths divisible by $5$ and hence
$|\Omega| \equiv 3$ modulo $5$. The only divisor of $|G|$ satisfying this property is $3$, but then $\Omega$
is too small.
The Sylow $3$-subgroups of $G$ have order $3$ and hence Lemma \ref{sylow}~(b) yields that
$x$ does not have order $3$. Thus $x$ is a $2$-element.

It follows from Lemma \ref{sylow}~(a) that $|\Omega|$ is odd, hence $G_\alpha$ has order $4$ or $12$ in the
$\Alt_5$-case and order $8$ or $24$ in the $\Sym_5$-case.
If $G \cong \Sym_5$ and $|G_\alpha|=24$, then this is the natural action of $\Sym_5$.

Assume that $G \cong \Alt_5$ and that $|G_\alpha|=12$. Then we first
note that every double transposition in $G$ has exactly three fixed
points on $\Omega$. As $|\Omega|=5$, there are only $10$
possibilities for fixed point sets for $x$. But there are 15 double
transpositions in $G$ and hence we find an involution $y \in G$ such
that $x \neq y$ and $\FO(x)=\FO(y)$. Then $x$ and $y$ interchange
the remaining two points and hence $xy$ fixes all of $\Omega$
point-wise, which is a contradiction.

Therefore, if $G_\alpha$ is not a $2$-group, then the only example
is $\Sym_5$ in its natural action. If $G_\alpha$ is a $2$-group,
then it is a Sylow $2$-subgroup and $G$ acts as stated.
\end{proof}

\begin{lemma}\label{A6}
Suppose that $G \cong \Alt_6$ and that $\Omega$ is a set such that $(G, \Omega)$ satisfies Hypothesis \ref{3fix}.
Then $|\Omega| \in \{6,15 \}$. The action of $G$ is
natural as $\Alt_6$ on six points in the first case, and $G$ acts as on the set of cosets of a subgroup of
order $24$ in the second case, respectively.
\end{lemma}

\begin{proof}
Let $\alpha \in \Omega$ and let $x \in G_\alpha$ be such that $|\FO(x)|=3$.
If $x$ is a $5$-element, then the subgroup structure of $G$ allows $G_\alpha$ to be of order $5$, $10$ or $60$.
However, this means that $|\Omega| \in \{72, 36, 6\}$ and these numbers are not congruent to $3$ modulo $5$.

Next suppose that $x$ is a $3$-element. Then Lemmas \ref{tuedel}~(c)
and \ref{sylow}~(b) imply that $G_\alpha$ has even order and that
$|\Omega|$ is divisible by $3$. Applying Lemma \ref{tuedel} to a
$2$-element in $G_\alpha$ yields that $|G_\alpha|$ is divisible by
$4$, hence by $12$. This leads to the cases $G_\alpha \cong \Alt_4,
\Sym_4$ or $\Alt_5$. Hence $|\Omega| \in \{30, 15, 6\}$. However the
first case is impossible as an element of order $3$ will fix six
points on $\Omega$. The other two possibilities give the examples in
the conclusion.

If $x$ is a $2$-element, then $G_\alpha$ has order $8$ or $24$ by
Lemma \ref{sylow}~(a). The former case leads to the possibility that
$|\Omega|=45$. However in this case an involution fixes five points,
which is impossible. The second case is that $G_\alpha \cong
\Sym_4$, which is one of our conclusions.
\end{proof}

\begin{lemma}\label{A6comp}
Suppose that $G$ is almost simple, but not simple and that
$F^*(G) \cong \Alt_6$. There does not exist a set $\Omega$ such that $(G,
\Omega)$ satisfies Hypothesis \ref{3fix}.
\end{lemma}

\begin{proof}
Let $E:=F^*(G)$. Then
Lemma \ref{hypcomp} is applicable and we see that (a) and (b) cannot hold. So (c) holds and we let
$\alpha \in \Omega$ be such that $(E, \alpha^E)$ satisfies Hypothesis \ref{3fix}.
In particular we know that $H:=E_\alpha \cong \Alt_5$ or $\Sym_4$ from Lemma \ref{A6}.

In the former case, $|\Omega| = 6$ or $12$ whereas in the second
case $|\Omega| = 15$ or $30$. If the action is on $6$ or $15$
points, then $G \cong S_6$ and one of the outer involutions has too
many fixed points.

If the action is on $12$ or $30$ points, then an inner involution
has four respectively six fixed points, ruling out these
possibilities as well.
\end{proof}

\begin{lemma}\label{A7}
Suppose that $G \cong \Alt_7$ and that $\Omega$ is a set such that
$(G, \Omega)$ satisfies Hypothesis \ref{3fix}. Then either
$|\Omega|=15$ and the action of $G$ is as on the set of cosets of a
subgroup isomorphic to $\PSL_2(7)$, or $|\Omega| =360$ and $G$ acts
on the set of cosets of a Sylow $7$-subgroup. In the first case the
three point stabilizer contains a Sylow $2$-subgroup of $G$.
\end{lemma}

\begin{proof}
Let $\alpha \in \Omega$ and $x \in G_\alpha$ be such that
$|\FO(x)|=3$. First assume that $x$ has order $7$. Then $|\Omega|
\equiv 3$ modulo $7$ and, as $|\Omega| \ge 4$, this only leaves the
possibilities $10$, $24$, $45$ or $360$. There are no subgroups of
$G$ of index $10$, $24$ or $45$, ruling out these cases. The
normalizer of a Sylow $7$-subgroup has index $360$ and this yields
the second example.

Next assume that $x$ has order $5$. Then Lemma \ref{tuedel} yields that some point stabilizer contains a
subgroup of order $20$, so we may suppose that $20$ divides $|G_\alpha|$. Moreover $|\Omega| \equiv 3$ modulo
$5$ and $7 \notin \pi(G_\alpha)$ by the subgroup structure of $\Alt_7$. In particular $7$ divides $|\Omega|$.
This only leaves the possibility $|\Omega|=63$ and $|G_\alpha|=40$.
But $G$ does not have a subgroup of this order.

We continue with the case where $x$ has order $3$.
Then Lemmas \ref{tuedel} and \ref{sylow}~(b) yield that $G_\alpha$ has even order and that $|\Omega|$ is
divisible by $3$.
From the centralizer of an involution in $G_\alpha$ and Lemma \ref{tuedel} we obtain that $G_\alpha$ contains
a subgroup isomorphic to $\Alt_4$. Thus $G_\alpha$ is isomorphic to $\Alt_4$, $\Sym_4$, $\Alt_5$, $\Sym_5$
or $\PSL_2(7)$. Correspondingly, $|\Omega| \in \{210, 105, 42, 21, 15\}$.

Assume that $G_\alpha \simeq \A_4$ and $|\Omega|=210$. Let $V \le G_\alpha$ be a fours group. Then $N_G(V)$
contains a subgroup of order $9$ of which a subgroup $A$ of order $3$ centralises $V$. Each involution in
$V$ has exactly two fixed points, hence $A$ fixes these two points and therefore $A \le G_\alpha$.
It follows that $9$ divides $G_\alpha$, which is contradiction.
With the same argument we exclude the case where $G_\alpha \simeq \A_5$ and $|\Omega|=42$.

Next assume that $G_\alpha \simeq \Sym_4$ and $|\Omega|=105$.
Then every involution has one or three fixed points. The second case will be treated below.
In the first case Lemma \ref{tuedel} forces $G_\alpha$ to contain a Sylow $2$-subgroup of $G$, which is
impossible.
With the same argument we exclude the case where $G_\alpha \simeq \Sym_5$ and $|\Omega|=21$.

Finally suppose that $x$ is a $2$-element. Then $G_\alpha$ contains a Sylow $2$-subgroup of $G$ by
Lemma \ref{sylow}~(a) and hence $3 \in \pi(G_\alpha)$ by Lemma \ref{tuedel}.
This means that $24$ divides $|G_\alpha|$ and the only new case is $G_\alpha \cong \Alt_6$.
But then $G$ acts as it does naturally on seven points; this is impossible because
one conjugacy class of $3$-elements has four fixed points in this action.

It follows that the only possibility is that $G_\alpha \cong \PSL_2(7)$.
Then Lemma \ref{tuedel} and the fact that $9$ does not divide $|G_\alpha|$ imply that
the three point stabilizer contains a Sylow $2$-subgroup of $G$.
\end{proof}

\begin{cor}\label{A7comp}
Suppose that $G \cong \Sym_7$. Then there is no set $\Omega$ is a set such that $(G,
\Omega)$ satisfies Hypothesis \ref{3fix}.
\end{cor}

\begin{proof}
Let $E:=F^*(G) \cong \A_7$. Then
Lemma \ref{hypcomp} is applicable and we see that (c) holds. Let $\alpha \in \Omega$ be such that
$(E, \alpha^E)$ satisfies Hypothesis \ref{3fix}. Then Lemma \ref{A7} yields that $E_\alpha \simeq \PSL_2(7)$
and that a Sylow $2$-subgroup of $E$ is contained in a three point stabilizer, or that $E_\alpha$ is a Sylow
$7$-subgroup of $G$.
In the first case, as $|G:E|=2$, Lemma \ref{tuedel} implies that a Sylow $2$-subgroup of $G$ is contained in a point stabilizer.
Therefore $|G_\alpha|=2 \cdot |E_\alpha|=2^4 \cdot 3 \cdot 7$. Let $t \in G_\alpha \setminus E_\alpha$ be
an involution. Then $|C_G(t)|$ is divisible by $5$, and
this contradicts Lemma \ref{tuedel} because $5 \notin \pi(G_\alpha)$.

In the second case, as $|G:E|=2$, Lemma \ref{tuedel} implies that
$G_\alpha$ contains an involution $t$. However then Lemma
\ref{tuedel} yiels that $C_G(t) \cap E_\alpha \neq 1$, contradicting
the fact that $E_\alpha \in \sy_7(G)$.

\end{proof}

\begin{lemma}\label{tollalt}
Suppose that Hypothesis \ref{3fix} holds and that $G$ is an
alternating group of degree at least $8$. Let $\alpha \in \Omega$.
Then $G_\alpha$ has odd order or it contains a Sylow $2$-subgroup of
$G$.
\end{lemma}

\begin{proof}
This follows immediately from Lemma \ref{2isttoll}.
\end{proof}

\begin{lemma}\label{A8}
Suppose that $G \cong \Alt_8$ and that $\Omega$ is a set such that
$(G, \Omega)$ satisfies Hypothesis \ref{3fix}. Then $|\Omega| = 2880$. The action of $G$ is as on the set of cosets of
a Sylow $7$-subgroup.
\end{lemma}

\begin{proof}
Let $\alpha \in \Omega$ and $x \in G_\alpha$ be such that
$|\FO(x)|=3$. First we suppose that $G_\alpha$ has odd order and we
choose $x$ of prime order $p$. Then $p \not \vdash 2$ by Lemma
\ref{key}. If $p=7$, then $|\Omega| \equiv 3$ modulo $7$ and
$|\Omega| \ge 5$, so in this case we only have the possibilities
that $|G_\alpha|=45$ or $|G_\alpha|=2880$. The former is impossible,
whereas the latter yields our example.

As $5 \vdash 2$ and $3 \vdash 2$, we see that $p \neq 5$ and $p \neq
3$, so this case is finished.

Using Lemma \ref{tollalt} we now have that $G_\alpha$ contains a
double transposition $t$. Then Lemma \ref{tuedel} yields that $32$
divides $|G_\alpha|$. Now we look at the normalizer of a fours group
in $G_\alpha$ and deduce that $3 \in \pi(G_\alpha)$. This gives two
possibilities: $G_\alpha$ is contained in a subgroup isomorphic to
$2^3:\PSL_3(2)$ or to $2^4:(\Sym_3 \times \Sym_3)$. Hence $|\Omega|
\in \{35, 105, 210\}$. However in all of these cases the involutions
have at least six fixed points, so this does not occur.
\end{proof}

\begin{cor}\label{A8comp}
Suppose that $G \cong \Sym_8$. Then there is no set $\Omega$ such
that $(G, \Omega)$ satisfies Hypothesis \ref{3fix}.
\end{cor}

\begin{proof}
Let $E:=F^*(G) \cong \A_8$. First we note that Lemma
\ref{hypcomp}~(c) holds and we let $\alpha \in \Omega$ be such that
$(E, \alpha^E)$ satisfies Hypothesis \ref{3fix}. Moreover $G_\alpha
\cap E$ is a Sylow $7$-subgroup of $G$.

Thus, as $|G:E|=2$, Lemma \ref{tuedel} implies that $G_\alpha$
contains an involution $t$. However then Lemma \ref{tuedel} implies
that $C_G(t) \cap E_\alpha \neq 1$, contradicting the fact that
$E_\alpha \in \sy_7(G)$.
\end{proof}

\begin{lemma}\label{A9}
Suppose that $G$ is isomorphic to $\Alt_9$ or $\Sym_9$. Then there is no set $\Omega$ such that
$(G, \Omega)$ satisfies Hypothesis \ref{3fix}.
\end{lemma}

\begin{proof}
First suppose that $G \cong \Alt_9$ and assume that $\Omega$ is a
set such that $(G, \Omega)$ satisfies Hypothesis \ref{3fix}. We let
$\alpha \in \Omega$ and begin as follows:

\smallskip

$(\ast)$ \hspace{1cm} $G_\alpha$ does not contain a $3$-cycle.

\begin{proof}
Assume otherwise. Then $G_\alpha$ contains an $\A_6$ (by Lemma
\ref{tuedel}), in particular Lemma \ref{tollalt} yields that $G$
contains involutions from both conjugacy classes. Then Lemma
\ref{tuedel} implies that $2^5 \cdot 3^3 \cdot 5$ divides
$|G_\alpha|$. But there is no maximal subgroup of $G$ that could
contain $G_\alpha$ now.
\end{proof}

Suppose first that $G_\alpha$ has odd order. Let $x \in G_\alpha$ be
of prime order $p$ and such that $|\FO(x)|=3$. We will use that $p
\not \vdash 2$ by Lemma \ref{key}. Then $p \neq 7$ because $7 \vdash
2$ and similarly $p \neq 5$. Hence $p=3$ and $x$ is not the product
of two $3$-cycles, by Lemma \ref{tuedel}. If $x$ is the product of
three $3$-cycles, then $x$ is $3$-central and therefore $G_\alpha$
contains a subgroup of order $3^3$. In particular $G_\alpha$
contains a $3$-cycle, contrary to $(\ast)$.

Lemma \ref{tollalt} yields that $G_\alpha$ contains a double
transposition. Applying Lemma \ref{tuedel} to its centralizer gives
that $G_\alpha$ has a subgroup isomorphic to $\A_5$, contrary to
$(\ast)$.

Now suppose that $G \cong \Sym_9$ and let $E:=F^*(G) \cong \Alt_9$.
Then by Lemma \ref{hypcomp} there is some $\alpha \in \Omega$ such
that $(E, \alpha^E)$ also satisfies Hypothesis \ref{3fix}. But this
is impossible by the previous paragraph.
\end{proof}

\begin{lemma}\label{A10}
Suppose that $G$ is isomorphic to $\Alt_{10}$ or $\Sym_{10}$.
Then there is no set $\Omega$ such that $(G, \Omega)$ satisfies Hypothesis \ref{3fix}.
\end{lemma}

\begin{proof}
Assume otherwise and let $\alpha \in \Omega$.

As in the previous lemma, we begin with the case where $G \cong
\A_{10}$. The special role of $3$-cycles will be key once more.

\smallskip

$(\ast)$ \hspace{1cm} $G_\alpha$ does not contain a $3$-cycle.

\begin{proof}
Assume otherwise. Then $G_\alpha$ contains a subgroup $H \cong \A_7$
(by Lemma \ref{tuedel}). In particular $|\Omega| \le 2^4 \cdot 3^2
\cdot 5$. Let $\beta \in \Omega$ be such that $\beta \neq \alpha$ and
let $\Delta:=\beta^H$. In its action on $\Delta$, every nontrivial
element of $H$ has at most two fixed points, and moreover $H$ does
not act regularly. But we proved in Lemma 3.5 in \cite{MW} that
$\A_7$ does not allow such an action. Hence this is impossible.
\end{proof}

Now we suppose that $G_\alpha$ has odd order and we let $x \in
G_\alpha$ be of prime order $p$. Then $p \not \vdash 2$ by Lemma
\ref{key}. In particular $p \neq 7$ and $p \neq 5$. If $p=3$, then
we first look at the case where $x$ is a product of three
$3$-cycles. Here $x$ is $3$-central and therefore $G_\alpha$
contains a subgroup of order $3^3$. In particular $G_\alpha$
contains a $3$-cycle, contrary to $(\ast)$.

If $x$ is the product of two $3$-cycles, then Lemma \ref{tuedel}
yields that $G_\alpha$ has even order, contrary to our assumption in
this case. By $(\ast)$ $x$ is not a $3$-cycle. So this case is
finished and by Lemma \ref{tollalt} it remains to consider the case
where $G_\alpha$ contains a Sylow $2$-subgroup of $G$. Then Lemma
\ref{tuedel}, applied to a double transposition, yields that
$G_\alpha$ contains an $\A_6$. But this is impossible by $(\ast)$.

If $G \cong \Sym_9$, then the previous paragraph and Lemma \ref{hypcomp} give the result.
\end{proof}

\begin{lemma} \label{no3}
Suppose that $n \ge 11$, that $G \cong \Sym_n$ or $\A_n$ and that
$\Omega$ is a set such that $(G,\Omega)$ satisfies Hypothesis
\ref{3fix}.

Then the order of a point stabilizer in $G$ is not divisible by $3$.
\end{lemma}

\begin{proof}
Assume otherwise and let $\alpha \in \Omega$. We show that our
hypothesis implies that $G_\alpha$ contains a $3$-cycle. Throughout
we use that, if $m \ge 5$, then $\A_m$ does not have subgroups of
index $2$ or $3$. This will play a role when applying Lemma
\ref{tuedel}.

We first note that $G_\alpha$ contains a $3$-cycle if it contains a
double transposition, by Lemma \ref{tuedel}, because the centralizer
of a double transposition in $G$ contains $\A_7$. Thus it is left to
prove our statement in the case where $G_\alpha$ has odd order, by
Lemma \ref{tollalt}.

Let $x \in G_\omega$ be an element of order $3$ and suppose that $k
\ge 2$ is such that $x$ is a product of $k$ cycles of length $3$. If
$n-3 \cdot k \ge 4$, then $C_G(x)$ contains a fours group or a
subgroup isomorphic to $\A_5$, which is impossible. Therefore $n- 3
\cdot k \le 3$. The structure of $C_{\A_n}(x)$ is $((3^k:\sym_k)
\times \sym_{n-3 \cdot k}) \cap \A_n$ and thus, if $k \ge 4$, then
again $C_G(x)$ contains a fours group. Thus $k \le 3$ and we obtain
that $n \le 3+3k \le 12$.

If $n=12$, then $C_G(x)$ contains a subgroup of structure
$((3^3:\sym_3) \times \sym_3) \cap \A_{12}$ and hence Lemma
\ref{tuedel} implies that $G_\alpha$ contains a $3$-cycle or a
double $3$-cycle. In the second case we change $x$ to such a double
$3$-cycle. Its centralizer contains an $\A_5$, so this is a
contradiction. If $n=11$, then $C_G(x)$ still contains a subgroup of
structure $(3^3:\sym_3)$ and thus, with Lemma \ref{tuedel}, it
follows once more that $G_\alpha$ contains a $3$-cycle.

As $G$ contains a subgroup isomorphic to $\A_{11}$, it is ninefold
transitive, and so we may suppose that $x = (1,2,3)$. It follows
from Lemma \ref{tuedel} that $G_\alpha$ contains a subgroup
isomorphic to $\A_{n-3}$ and hence, without loss, the $3$-cycle
$y:=(4,5,6)$. The same argument yields that $C_G(y) \le G_\alpha$. Now
we deduce that $G_\alpha \geq \< C_G(x), C_G(y)\> \cong \A_n$, which
contradicts the fact that $G$ acts faithfully on $\Omega$.
\end{proof}

\begin{thm}\label{bigalt}
Suppose that $n \geq 11$ and that $G \cong \A_n$ or $\Sym_n$. Then
there is no set $\Omega$ such that $(G,\Omega)$ satisfies Hypothesis
\ref{3fix}.
\end{thm}

\begin{proof}
Assume that $\Omega$ is a set such that $(G,\Omega)$ satisfies
Hypothesis \ref{3fix}. Let $\alpha \in \Omega$, let $p$ be a prime
and let $x \in G_\alpha$ be a $p$-element. Then there exists $k \in
\N$ such that $x$ is a product of $k$ cycles of length $p$. Now
$C_G(x)$ contains a subgroup of structure $p^k: \Sym_k  \times
\A_{n-p \cdot k}$ if $p$ is odd and of structure $(2^k: \Sym_k
\times \Sym_{n-2\cdot k})\cap \A_n$ otherwise.

Assume that
$n-p \cdot k \geq 3$. Then Lemma
\ref{tuedel} implies that $C_G(x) \cap G_\alpha$ contains a
$3$-cycle, contrary to Lemma \ref{no3}.

Therefore $n-p \cdot k \le 2$, so $11 \le n \le 2+p \cdot k$. First
we assume that $p = 2$. Then $G_\alpha$ contains a double
transposition $t$ by Lemma \ref{tollalt}. As $n \geq 11$, we see
that $C_G(t)$ contains a subgroup isomorphic to $\A_7$, which is a
perfect group of order divisible by $3$. Together with Lemma
\ref{tuedel} this contradicts Lemma \ref{no3}. This means that
$G_\alpha$ has odd order.

With Lemma \ref{no3} it follows that $p > 3$. Then Lemma \ref{sylow}~(c)
implies that $G_\alpha \cap \A_n$ contains a full Sylow $p$-subgroup
$P$ of $G$. Thus $G_\alpha \cap \A_n$ contains a $p$-cycle, say $y$.
If $n-p>3$, then $C_{G_\alpha}(y)$ contains a double transposition by Lemma
\ref{tuedel}, and this contradicts the fact that $G_\alpha$ has odd order.

Therefore $n-p \leq 3$ and this property holds for all prime divisors
$p$ of $|G_\omega|$.

As $n \geq 11$, the above property forces $p \geq 8$. But $p$ is prime and so we have that $p \geq 11$. In
particular $|N_G(\langle x \rangle):\langle x \rangle| \geq
\frac{p-1}{2} \geq 5$ and it follows that $|G_\omega \cap
N_G(\langle x \rangle)|$ is divisible by a prime $r$ such that $2 \cdot r
\leq p-1 \leq n$. This implies $r \leq n-r$. We know that $r \neq 2$
and $r \neq 3$ (by Lemma \ref{no3} and because $G_\alpha$ has odd order), so $5 \leq r \leq n-r$.
We proved in the previous paragraph
that $r$ satisfies $n-r \leq 2$. Now this is impossible.
\end{proof}

The next result collects all the information from this chapter.

\begin{thm}\label{altmain}
Let $n \in \N$ and suppose that $G$ is isomorphic to $\A_n$ or to $\Sym_n$.
If $\Omega$ is a set such that $(G, \Omega)$ satisfies Hypothesis \ref{3fix}, then
$n \in \{5,6,7,8\}$ and one of the following holds:

\begin{enumerate}

\item[(1)]
$n=5$, $G \cong \A_5$,
$|\Omega|=15$ and the action of $G$ is as on the set of cosets of a Sylow $2$-subgroup.

\item[(2)]
$n=5$, $G \cong \Sym_5$, $|\Omega|=5$ and $G$ acts naturally.

\item[(3)]
$n=6$, $G \cong \A_6$, $|\Omega| = 6$ and $G$ acts naturally.

\item[(4)]
$n=6$, $G \cong \A_6$, $|\Omega| = 15$ and $G$ acts as on the set of cosets of a subgroup of order $24$.

\item[(5)]
$n=7$, $G \cong \A_7$, $|\Omega|=15$ and the action of $G$ is
as on the set of cosets of a subgroup isomorphic to $\PSL_2(7)$.

\item[(6)]
$n=7$, $G \cong \A_7$, $|\Omega|=360$ and the action of $G$ is
as on the set of cosets of a Sylow $7$-subgroup.

\item[(7)]
$n=8$, $G \cong \A_8$, $|\Omega|=2880$ and the action of $G$ is
as on the set of cosets of a Sylow $7$-subgroup.

\end{enumerate}
\end{thm}

\begin{proof}
Theorem \ref{bigalt} and Lemma \ref{S4} imply that $n \in \{5,6,7,8\}$.
Moreover $\Sym_7$ and $\Sym_8$ do not occur by Lemmas \ref{A6comp}, \ref{A7comp} and \ref{A8comp}.

The possibilities are then listed in Lemmas \ref{S5}, \ref{A6},  \ref{A7} and \ref{A8}.
\end{proof}

\section{Lie type groups}

We organize our analysis around Lemma \ref{2isttoll} and begin with
the almost simple groups where the normalizers of Sylow
$2$-subgroups are strongly $2$-embedded. Then we consider groups
with dihedral or semidihedral Sylow $2$-subgroups and finally those
groups where we know from the outset that $|G_\omega|$ is odd.

We record a general lemma, which is a consequence of Lemma
\ref{charfrob} for use in this section.

\begin{lemma}\label{2or3}
Suppose that $(G,\Omega)$ satisfies Hypothesis \ref{3fix} with
$|\Omega| \geq 7$ and suppose that $\alpha, \beta, \gamma \in
\Omega$ are pair-wise distinct and such that $1 \neq H := G_\alpha
\cap G_\beta \cap G_\gamma$. Then there exists a subgroup $1  \neq X
\leq H$ and an element $g \in N_G(X) \setminus X$ such that $3$
divides $o(g)$, or $|G_\alpha|$ has even order.
\end{lemma}

\begin{proof}
The nonidentity subgroups $X$ of $H$ fix the elements of
 $\Delta:=\{\alpha,\beta,\gamma\}$ and act semiregularly on $\Omega
\setminus \Delta$.  Thus for every such $X$ we see that $N_G(X)$
acts on $\Delta$ with kernel $N_H(X)$, and $|N_G(X):N_H(X)| \le 3$
by Lemma \ref{tuedel}~(c).

Suppose, for all nontrivial subgroups $X$ of $H$, that
$(|N_G(X):N_H(X)|,3) =1$.

If $1 \neq X \leq H$ is such that $|N_G(X):N_H(X)| = 2$, then
$N_G(X)$ has even order and a fixed point on $\Delta$ which is one
of our conclusions.

If $H$ has no nontrivial subgroup $X$ such that $|N_G(X):N_H(X)| =
2$, then for all these subgroups $N_G(X) \le H$. Since $1 \neq H
\neq G$, it follows with Lemma \ref{charfrob} that $G$ is a
Frobenius group. But this contradicts Hypothesis \ref{3fix}.
\end{proof}

\subsection{Groups with strongly embedded Sylow $2$-subgroup normalizers}

The simple groups of Lie type considered in this section are those
where the normalizers of Sylow $2$-subgroups are strongly embedded.
We consider them in individual lemmas.

\begin{lemma} \label{PSL2e}
Suppose that $q$ is power of $2$ and that $G = \PSL_2(q)$. Then
there is no set $\Omega$ such that $(G, \Omega)$ satisfies
Hypothesis \ref{3fix}.
\end{lemma}

\begin{proof}
We note that $\PSL_2(4) \cong \PSL_2(5) \cong \Alt_5$ and that
$\PSL_2(2) \cong \Sym_3$. These groups never satisfy Hypothesis
\ref{3fix}, so we may suppose that $q \geq 8$. We assume that the
lemma is false and let $\Omega$ be such that $(G, \Omega)$ satisfies
Hypothesis \ref{3fix}. Let $\omega \in \Omega$.

First we suppose that $|\Omega|$ is odd. Then $G_\omega$ contains a
Sylow $2$-subgroup $S$ of $G$. Now $|N_G(S)/S| = q-1 \geq 7 > 3$, so
Lemma \ref{tuedel} implies that $G_\omega$ contains an element $x$
of order $(q-1)/(q-1,3)$. If $(q-1,3) = 1$, then $G_\omega =
N_G(S)$, as $N_G(S)$ is maximal in $G$. If $(q-1,3) \neq 1$, then we
note that $N_G(\langle x \rangle)$ is dihedral of order $2(q-1)$.
Lemma \ref{tuedel} implies that either $|G_\omega \cap N_G(\langle x
\rangle)| = q-1$, in which case $G_\omega = N_G(S)$, or $|G_\omega
\cap N_G(\langle x \rangle)| = 2(q-1)/3$, in which case $G_\Omega$
contains an involution which does not lie in $S$. As $S$ together
with any involution $t \not \in S$ generates $G$, we see that the
latter cannot happen and that $G_\omega = N_G(S)$. Thus $(G,\Omega)$
appears in the conclusion of the main theorem of \cite{MW}, and in
particular no nonidentity element of $G$ has three fixed points on
$\Omega$. This is a contradiction.

Thus we may now suppose that $|\Omega|$ is even. If $S \in
\sy_2(G)$, then $S$ is elementary abelian of order at least $8$ and
thus Lemma \ref{2isttoll} implies that $G_\omega$ has odd order.
Inspection of the maximal subgroups of $G$ yields that $G_\omega$ is
cyclic of order dividing $q-1$ or $q+1$. This means that, if $x \in
G_\omega$, then $|F_\Omega(x)| = |N_G(\< x \>):G_\omega| \geq
|N_G(G_\omega):G_\omega|  \in \{2 \cdot \frac{q-1}{|G_\omega|}, 2
\cdot \frac{q+1}{|G_\omega|}\}$ and hence $|F_\Omega(x)| \leq 3$.
Now $|G_\omega| \in \{q-1,q+1\}$ and so $(G, \Omega)$ appears in the
conclusion of the main theorem of \cite{MW}. In particular no
nonidentity element of $G$ has three fixed points on $\Omega$,
contrary to our assumption.
\end{proof}

As a corollary of Lemma \ref{3strich} we obtain:

\begin{lemma} \label{Szq}  Let $q$ be a power of $2$, $q \geq 8$, and suppose that
$G = \Sz(q)$. Then there is no set $\Omega$ such that $(G, \Omega)$
satisfies Hypothesis \ref{3fix}.
\end{lemma}

Prior to proving our next lemma we note that the group $\PSU_3(2)$ is a Frobenius group of order $72$, and in particular it does not lead to any examples for Hypothesis \ref{3fix}.

\begin{lemma} \label{PSU3e}  Let $q \ge 4$ be a power of $2$ and let
$G = \PSU_3(q)$.

Let $\Lambda$ be the set of cosets of a cyclic subgroup of order
$q^2-q+1/(3,q+1)$ of $G$. Then
$(G, \Lambda)$ satisfies Hypothesis \ref{3fix},
and this is the unique example for $G$.
\end{lemma}

\begin{proof}
Let $\Omega$ be such that $(G, \Omega)$ satisfies Hypothesis \ref{3fix}. We
show that the point stabilizers are cyclic of order $q^2-q+1/(3,q+1)$ and that the action described in the lemma does in fact give an example.

Let $\omega \in \Omega$. We first consider the situation where $|\Omega|$ is odd; i.e.
$S \leq G_\omega$ for some $S \in \sy_2(G)$.

The group $N_G(S)/S$ is cyclic of order $q^2-1$, which implies by
Lemma \ref{tuedel} that $G_\omega$ also contains a subgroup of order
$((q+1)/(q+1,3))^2 \neq 1$. However $N_G(S)$ is strongly embedded in
$G$, so the proper overgroups of $S$ in $G$ are contained in
$N_G(S)$. Thus $G_\omega = G$, which is impossible.
Now $|\Omega|$ is even.

Next we note that $S$ is neither dihedral nor semidihedral, so Lemma
\ref{2isttoll} implies that $G_\omega$ has odd order. The elements
of $G$ of odd order are conjugate to elements of tori of orders
$q^2-1$, $(q+1)^2/(q+1,3)$ or $(q^2 - q + 1)/(q+1,3)$.

First suppose that $p \in \pi(G_\omega)$ is such that $p$ divides $q
- 1$. Then $p \vdash r$ for all divisors $r$ of $q+1/(3,q+1)$ and
Lemma \ref{key} yields that all these primes $r$ divide
$|G_\omega|$. Thus if $p \in \pi(G_\omega)$ divides $(q^2-1)$, then
Lemma \ref{tuedel} implies that $((q+1)/(q+1,3))^2 $ divides
$|G_\omega|$. This means that $G_\omega$ contains an element $y$
with $C_G(y)$ of structure $(q+1) \times \PSU_2(q)$ and hence
$G_\omega$ contains a subgroup isomorphic to $\PSU_2(q)$,
contradicting the fact that $|G_\omega|$ is odd.

Thus no $p \in \pi(G_\omega)$ divides $q^2-1$, which implies that
all $p \in \pi(G_\omega)$ divide $q^2-q+1 / (3,q+1)$.  Now if $x \in
G_\omega$ has prime order $p$, then $C_G(x)$ is cyclic of order
$q^2-q+1 / (3,q+1)$, and $|N_G(\langle x \rangle):C_G(x)| = 3$. As
$3$ divides $q^2-1$, but not $|G_\omega|$, this yields that $G_\omega
= C_G(x)$.\\

The previous arguments show that there is at most one possibility for the action of $G$ on $\Omega$.
Now let $\Lambda$ be the set of cosets of $C_G(x)$ in $G$.
We show that this actually gives an example.
Since $(q+1,q^3+1) = 3(q+1)$, we see that $C_G(y) =
C_G(x)$ for all $y \in C_G(x)^\#$ and thus $|\FL (y)| = |N_G(\<
y \>):C_G(x)| = |N_G(C_G(x)):C_G(x)| = 3$. This shows that
$(G,\Lambda)$ satisfies Hypothesis \ref{3fix} as claimed.
\end{proof}

\subsection{Groups with dihedral or semidihedral Sylow 2-subgroups}

The simple groups of Lie type considered in this section are those
whose Sylow $2$-subgroups are dihedral or semidihedral. Again we
look at the corresponding series of groups in individual lemmas.

\begin{lemma} \label{PSL2o}
Suppose that $q$ is a power of an odd prime and that $G =
\PSL_2(q)$. Then $(G, \Omega)$ satisfies Hypothesis \ref{3fix} if
and only if one of the following is true:

\begin{enumerate}
\item  $G \cong \PSL_2(7) \cong \PSL_3(2)$ with $|\Omega| = 7$ and
$G_\omega \cong \Sym_4$.
 \item $G \cong \PSL_2(7) \cong \PSL_3(2)$ with $|\Omega| = 24$ and
 $G_\omega$ is cyclic of order $7$.
\item  $G \cong \PSL_2(11)$ with $|\Omega| = 11$ and $G_\omega \cong \Alt_5$.
\end{enumerate}

\end{lemma}

\begin{proof}
We note that $\PSL_2(5) \cong \Alt_5$, $\PSL_2(9) \cong \Alt_6$ and
that $\PSL_2(3)$ solvable. Therefore we may assume that $q = 7$ or
$q \geq 11$. We also assume that $(G, \Omega)$ satisfies Hypothesis
\ref{3fix}.

The full table of marks of $\PSL_2(7)$ and $\PSL_2(11)$ is available
in GAP (see \cite{GAP}) and these confirm our claim. Thus we may
assume that $q \geq 13$. Let $\omega \in \Omega$.

If $r \in \pi(G_\omega)$ is a divisor of $(q+1)/2$ and if $x \in
G_\omega$ has order $r$, then $N_G(\<x\>)$ is dihedral of order
$q+1$. As in the proof of Lemma \ref{PSL2e} this implies that
$G_\omega$ is cyclic of order $(q \pm 1)/2$. This action occurs in
the conclusion of the main theorem of \cite{MW}, contradicting
Hypothesis \ref{3fix}.

Now if $r \in \pi(G_\omega)$ and $(q,r) \neq 1$, then $r \vdash p$
for all divisors $p$ of $(q-1)/2$. Thus we assume that $G_\omega$
contains an element $x$ of order dividing $(q-1)/2$. As $N_G(\<x \>
)$ is dihedral of order $(q-1)$, Lemma \ref{tuedel} implies that
$G_\omega$ contains a subgroup of index at most $3$ of this
normalizer. Now assume that $G_\omega$ contains an involution $t$
inverting $x$. Then $G_{G_\omega}(x)$ and $C_{G_\omega}(t)$ generate
$G$ (by the subgroup structure of $G$). This is impossible.
The only overgroups of $\<x\>$ are conjugates of $B$,
the Borel subgroup of $G$, or the dihedral group of order $(q-1)$.
The latter possibility is ruled out because no
involution in $G_\omega$ inverts $x$, and the possibility $G_\omega = B$ is ruled out because the action of $G$ on the set of
cosets of $B$ occurs in the conclusion of the main theorem of
\cite{MW}. Thus $G_\omega$ is cyclic of order $(q-1)/2$
but again this possibility occurs in the conclusion of the main
theorem of \cite{MW}. This proves that $\PSL_2(q)$ for $q \geq 13$
does not yield examples satisfying Hypothesis \ref{3fix}.
\end{proof}

\begin{lemma} \label{PSU3o}
Suppose that $p$ is an odd prime, that $q=p^a$ with $a \in \N$ and
that $G = \PSU_3(q)$.
Let $\Lambda$ be the set of cosets of a cyclic subgroup of order
$q^2-q+1/(3,q+1)$ of $G$. Then
$(G, \Lambda)$ satisfies Hypothesis \ref{3fix},
and this is the unique example for $G$.
\end{lemma}

\begin{proof}
Let $\Omega$ be such that $(G, \Omega)$ satisfies Hypothesis \ref{3fix}. We
show that the point stabilizers are cyclic of order $q^2-q+1/(3,q+1)$ and that the action described in the lemma does in fact give an example.

Let $\omega \in \Omega$. For $q=3$ our claim follows
from inspection of the table of marks in GAP (see \cite{GAP}). So we
may suppose that $q \geq 5$.

If $t \in G_\omega$ is an involution, then $C_G(t)$ contains a
subgroup isomorphic to $\SL_2(q)$. But $\SL_2(q)$ is perfect because
$q \ge 5$, and therefore Lemma \ref{tuedel} implies that $G_\omega$
has a subgroup isomorphic to $\SL_2(q)$. Let $P \le G_\omega$ be
such that $P$ is isomorphic to a Sylow $p$-subgroup of $\SL_2(q)$.
Then $G_\omega$ contains an index three subgroup of $N_G(P)$, again
by Lemma \ref{tuedel}.

If $p > 3$, then Lemma \ref{sylow}~(c) implies that $G_\omega$
contains a Sylow $p$-subgroup of $G$, and if $p = 3$, then a
straightforward computation shows that a torus in $N_G(P)$ acts
transitively on the commutator factor group of $\sy_p(N_G(P))$. In
this case $G_\omega$ contains a Sylow $p$-subgroup of $N_G(P)$, and
hence of $G$, again. This is impossible because a subgroup of $G$
isomorphic to $\SL_2(q)$ together with a Sylow $p$-subgroup
generates all of $G$.

So we may now suppose that $|G_\omega|$ is odd. If $p \in
\pi(G_\omega)$, then $p \vdash r$ for every prime divisor $r$ of
$q^2-1/(9,q^2-1)$ and hence Lemma \ref{key} implies that all these
primes $r$ divide $|G_\omega|$. From the existence of tori of order
$(q^2-1)/(3,q+1)$ and $(q+1)^2/(3,q+1)$ it follows that
$(\frac{(q+1)}{(3,q+1)})^2$ divides $|G_\omega|$, whenever $p$ or a
divisor of $q^2-1$ divides $|G_\omega|$. Thus there exist commuting elements
$x_1,x_2 \in G_\omega$ with centralizers containing a subgroup
isomorphic to $\SL_2(q)$
and such that $G = \<C_G(x_1)', C_G(x_2)'\>$. However, Lemma \ref{tuedel} then forces $G= G_\omega$. This is a
contradiction.

Thus no prime $p \in \pi(G_\omega)$ divides $q^2-1$. This means that
they all divide $q^2-q+1 / (3,q+1)$. Let $x \in G_\omega$ be of
prime order $p$. Then $C_G(x)$ is cyclic of order $q^2-q+1 /
(3,q+1)$, and $|N_G(\langle x \rangle):C_G(x)| = 3$. But $3$ does
not divide $G_\omega$, so this implies that $G_\omega = C_G(x)$. \\

These arguments show that there
is at most one possibility for the action of $G$ on $\Omega$.
Now let $\Lambda$ be the set of cosets of $C_G(x)$ in $G$.

As
$(q+1,q^3+1) = 3(q+1)$ it follows that $C_G(y) = C_G(x)$ for all $y
\in C_G(x)^\#$ and thus $|\FL (y)| = |N_G(\< y \>):C_G(x)| =
|N_G(C_G(x)):C_G(x)| = 3$ which shows that $(G,\Lambda)$
satisfies Hypothesis \ref{3fix}.
\end{proof}

\begin{lemma} \label{PSL3o}  Let $G = \PSL_3(q)$  with $q$ odd. If
$(G, \Omega)$ satisfies Hypothesis \ref{3fix}, then for all $\omega
\in \Omega$ the group $G_\omega$ is cyclic of order $(q^2 + q + 1) /
(3,q-1).$ Moreover $|N_G(G_\omega)| = 3 \cdot |G_\omega|$ and
$(|G_\omega|,3) = 1$.
\end{lemma}

\begin{proof} Inspection of the table of marks in GAP establishes our
claim for $q=3$.
Thus we may assume that $q \geq 5$. Let $\omega \in \Omega$.

If $r \in \pi(G_\omega)$ and $r$ is a divisor of $q(q^2-1)$, then $r
\vdash s$ for all prime divisors $s$ of $q-1$. Thus in every such
case a subgroup of index at most $3$ of a split torus $T$ of order
$(q-1)^2/(3,q-1)$ will be contained in $G_\omega$. But this implies,
as in the proof of Lemma \ref{PSU3o}, that $G_\omega$ has
commuting elements $x_1, x_2$ with centralizers containing a subgroup isomorphic to
$\SL_2(q)$ and so that $G = \<C_G(x_1)', C_G(x_2)'\>$. But then Lemma \ref{tuedel} forces  $G=G_\omega$, which is a contradiction.

Thus the only possibilities for $r \in \pi(G_\omega)$ are divisors
of $(q^2+q+1)/(3,q-1)$. If $x \in G_\omega$ has order $r$
dividing $(q^2+q+1)/(3,q-1)$, then $C_G(x)$ is cyclic of order
$q^2+q+1 / (3,q+1)$, and $|N_G(\langle x \rangle):C_G(x)| = 3$. As
$3$ divides $p(q^2-1)$, but not $G_\omega$, this implies that
$G_\omega = C_G(x)$. Moreover $(q-1,q^3-1) = 3(q-1)$ and so we see
that $C_G(y) = C_G(x)$ for all $y \in C_G(x)^\#$. Thus
$|F_\Omega(y)| = |N_G(\< y \>):G_\omega| = |N_G(G_\omega):G_\omega|
= 3$. This shows all assertions of the lemma.
\end{proof}

\subsection{Point Stabilizers of odd order}

The groups treated in the previous sections were those whose Sylow
$2$-subgroups fell into conclusions (2) or (3) of Lemma
\ref{2isttoll}. In what follows, we therefore work under the
following hypothesis:

\begin{hyp}\label{newlie}
$(G,\Omega)$ satisfies Hypothesis \ref{3fix}. Moreover $G$ is a
simple group of Lie type, but neither $\PSL_2(q)$, $\Sz(q)$ or
$\PSU_3(q)$ where $q$ is even, nor $\PSL_2(q)$, $\PSU_3(q)$ or
$\PSL_3(q)$ where $q$ is odd. Moreover $|G_\omega|$ has odd order.
\end{hyp}

\begin{lemma}\label{Sp43}
If $(G,\Omega)$ satisfies Hypothesis \ref{newlie}, then $G \not \cong \Sp_4(3)$.
\end{lemma}
\begin{proof}
We first observe that $5 \vdash 2$ and thus $G_\omega$ is a
$3$-group by Lemma \ref{key}. The centralizers of elements of order
$3$ in $G$ have order divisible by $27$, so Lemmas \ref{tuedel} and
\ref{syl3} imply that $|G_\omega| \geq 27$. Moreover the Sylow
$3$-subgroups of $G$ are isomorphic to $3 \wr 3$, therefore we see
that $G_\omega$ contains $3$-central elements whose centralizer
order is divisible by $4$. Together with Lemma \ref{tuedel} this
contradicts Hypothesis \ref{newlie}.
\end{proof}

\begin{lemma}\label{3musssein}
Suppose that $(G,\Omega)$ satisfies Hypothesis \ref{newlie} and let
$\alpha, \beta, \gamma \in \Omega$ be pair-wise distinct and such
that $1 \neq H := G_\alpha \cap G_\beta \cap G_\gamma$. Then
$|N_G(X):N_H(X)| \in \{1,3\}$ for all $1 \neq X \in H$ and there
exists a nontrivial subgroup $X$ of $H$ such that $|N_G(X):N_H(X)| =
3$.
\end{lemma}

\begin{proof}
Hypothesis \ref{newlie} implies that $|\Omega| \geq 7$. For all
nontrivial subgroups $X$ of $H$, we know by Lemma \ref{tuedel} that
$|N_G(X):N_H(X)| \le 3$. There exists some $1 \neq X \le H$ such
that $N_G(X) \nleq H$ by Lemma \ref{charfrob}, and for this subgroup
$|N_G(X):N_H(X)| \in \{2,3\}$. However, index $2$ cannot occur
because otherwise some $2$-element in $N_G(X)$ fixes one of $\alpha,
\beta,\gamma$, contrary to Hypothesis \ref{newlie}.
\end{proof}

We recall that, in a simple group $G$ of Lie type of characteristic
$p$, an element $g$ is called semisimple if and only if its order is
coprime to $p$. A semisimple element is called regular semisimple if
and only if $(|C_{G}(g)|,p) = 1$. We note that the centralizer of a
nonregular semisimple element contains a subgroup which is
isomorphic to either $\SL_2(q)$ or $\PSL_2(q)$ and is generated by
root elements of $G$. Recall that $\SL_2(q)$ is a perfect group when
$q \geq 4$, and hence does not contain subgroups of index less than
or equal to $4$. Moreover $\SL_2(3) \cong Q_8:3$ and $\SL_2(2) =
\PSL_2(2) \cong \Sym_3$.

\begin{lemma} \label{general}
If $(G,\Omega)$ satisfies Hypothesis \ref{newlie}, then all
non-identity elements in point stabilizers are regular semisimple
elements.

\end{lemma}

\begin{proof}
Let $\omega \in \Omega$ and suppose that some non-identity element
$g \in G_\omega$ is not regular and semisimple. Then either $g$
is semisimple and $C_G(g)$ contains a subgroup isomorphic to
 $\SL_2(q)$ or $\PSL_2(q)$ and is generated by root elements of $G$, or
$g$ is not semisimple.

If $g \in G_\omega$ is not semisimple, then $g$ powers to a
$p$-element, so Hypothesis \ref{newlie} implies that $p$ is odd. If
$p \neq3 $, then Lemma \ref{sylow} yields that $G_\omega$ contains a
full Sylow $p$-subgroup of $G$. Thus $G_\omega$ contains a long root
element $r$ and also $C_G(r)$ which, under Hypothesis \ref{newlie},
is a perfect group containing an $SL_2(q)$. Thus $G_\omega$ has even
order which is a contradiction to Hypothesis \ref{newlie}.

If $p=3$ and $G_\omega$ contains a $3$-element, then Lemma
\ref{syl3} yields that $G_\omega$ either contains an index $3$
subgroup of a Sylow $3$-subgroup of $G$, or the Sylow $3$-subgroup
of $G$ is of maximal class. The latter is excluded by Hypothesis
\ref{newlie}, so we may assume that $G_\omega$ contains an index $3$
subgroup of a Sylow $3$-subgroup $P$ of $G$. The Chevalley
commutator relations imply that any index $3$ subgroup of $P$
must contain $Z(P)$ and thus $Z(P) \leq
G_\omega$.  If $G \not \cong ^2G_2(q)$, then it follows from Lemma
\ref{tuedel} that an index $3$ subgroup of $C_G(Z(P))$ is contained
in $G_\omega$. But $|C_G(Z(P))|_2 \geq 8$, so now $|G_\omega|$ has
even order, contradicting Hypothesis \ref{newlie}. Finally if $G
\cong ^2G_2(q)$, then $N_G(Z(P))$ has structure $P:(q-1)$ which
implies that an element $h$ of order $(q-1)/2$ lies in $G_\omega$.
Now $|N_G(\< h \>)|_2 = 4$ which by Lemma \ref{tuedel} forces
$|G_\omega|$ to be even, again contradicting Hypothesis
\ref{newlie}.

Thus we have shown that the elements of $G_\omega$ are semisimple.

If $q \geq 4$ and $g$ is semisimple, but not regular, then $C_G(g)$
contains a subgroup isomorphic to
 $\SL_2(q)$ or $\PSL_2(q)$ which is perfect. Hence Lemma \ref{tuedel} forces
 $|G_\omega|$ to be even, which
violates Hypothesis \ref{newlie}.

If $q =3$ and $g$ is semisimple, but not regular, then $C_G(g)$
contains a subgroup isomorphic to $\SL_2(3)$.
As $|\SL_2(3)|_2 = 8$, Lemma \ref{tuedel} implies that
$|G_\omega|_2 \neq 2$, contradicting Hypothesis \ref{newlie}.

If $q =2$, and $g$ is semisimple, but not regular, then $C_G(g)$ has
a subgroup isomorphic to $\SL_2(2)$. Since this group has order $6$,
Lemma \ref{tuedel} shows that $(|G_\omega|,6) \neq 1$. So under
Hypothesis \ref{newlie} this means that $G_\omega$ contains a
$3$-element whose centralizer contains the centralizer of a root
subgroup $\SL_2(2)$, say $R$. Hypothesis \ref{newlie} implies that
$G$ is not of rank $2$, because $\PSL_3(2) \cong \PSL_2(7)$,
$\Sp_4(2) \cong \PSL_2(9) \cong \Alt_6$, $\PSU_4(2) \cong
\PSp_4(3)$,
 and $G_2(2)' \cong \PSU_3(2)$.
As $G$ has rank at least $3$, we see that $|R|_2 \geq 4$, which
implies that $|G_\omega|_2 \neq 2$. Again this contradicts
Hypothesis \ref{newlie}.

\end{proof}

Having established that every element in a point stabilizer is
regular, we now consider centralizers of regular semisimple elements
in groups satisfying Hypothesis \ref{newlie}. We note that such a
centralizer is a torus.
Moreover the order of a torus is a polynomial in $q$ of degree equal
to the untwisted Lie rank of $G$.

\begin{lemma}\label{2G2}
If $q$ is a prime power, $q>3$ and $G=^2G_2(q)$, then there is no
set $\Omega$ such that $(G,\Omega)$ satisfies Hypothesis \ref{3fix}.
\end{lemma}

\begin{proof}
Assume otherwise and let $\omega \in \Omega$. If $g \in
G_\omega^\#$, then $|C_G(g)| = q-1$, $q+ \sqrt{3q} +1$ or $q-
\sqrt{3q} +1$.

If $|C_G(g)| = q-1$, then $|N_G(\<g\>)|_2 = 4$ which implies that
$|G_\omega|$ is even, a contradiction to Hypothesis \ref{newlie}. If
$|C_G(g)| = q+ \sqrt{3q} +1$ or $q- \sqrt{3q} +1$, then Lemma
\ref{tuedel} implies that $C_G(g) \leq G_\omega$. Next we recall
that $|N_G(C_G(g))/C_G(g)| = 6 $, so Lemma \ref{tuedel} yields that
$2$ or $3$ divides $|G_\omega|$. The former contradicts Hypothesis
\ref{newlie} whereas the latter contradicts Lemma \ref{general}.
\end{proof}

\begin{hyp}\label{betterlie}
From now on until the end of this subsection we suppose that
$(G,\Omega)$ satisfies Hypothesis \ref{newlie} and that $G$ is of
Lie rank at least $2$, but not isomorphic to
$\PSL_3(2),G_2(2),\Sp_4(2),\PSU_4(2) \cong \Sp_4(3)$ or $\PSL_4(2)
\cong \Alt_8$. Moreover all nonidentity elements of $G_\omega$ are
regular and semisimple.
\end{hyp}

We denote the natural module of $G$ by $N$. By $\phi_d(x)$ we denote
the irreducible cyclotomic polynomial dividing $x^d-1$, but not
$x^k-1$ for all $k < d$.

\begin{lemma}\label{torus}
Suppose that $(G,\Omega)$ satisfies Hypothesis \ref{betterlie} and
let $\omega \in \Omega$.  If $g \in G_\omega^\#$, then the following
are true:
\begin{enumerate}
\item  $C_G(g)$ is a maximal torus of $G$.

\item If $G$ is a classical group, then $\dim(N) - \dim([N,g^i]) \leq 2$
for all $ i < o(g)$. Moreover if $\phi_d(q)$ is a divisor of
$|C_G(g)|$, then $|N_G(T)/T|$ is divisible by $d$.

\item If $G$ is a classical group, then $(3,|G_\omega|) = 1$.

\item If $G$ is an exceptional group and $C_G(g)$ is not a $3$-group,
then $4$ divides $|N_G(C_G(g))/C_G(g)|$.

\item If $G$ is an exceptional group, then for all
$3$-elements the centralizer has order divisible by $8$.

\end{enumerate}
\end{lemma}

\begin{proof}
The conclusion of Lemma \ref{general} is that $g$ is regular and
semisimple. This implies that $C_G(g)$ is a maximal torus; i.e. (1)
follows.

Suppose now that $G$ is classical and that $C_G(g)$ is not cyclic.
Recall that $N$ denotes the natural module of $G$. If $\phi_d(q)$ is
a divisor of $|G|$, then $G$ possesses an element $x_d$ such that
$\dim([N,x_d]) \in \{d,2d\}$ and $[N,x_d]$ is nondegenerate with
respect to the form defining $G$.

This is clear if  $G = {\rm SL}(N)$ as there exists a $d \times d$
matrix with characteristic polynomial $\phi_d(x)$. For ${\rm Sp}(N)$
and ${\rm SU}(N)$ we embed the element via the overfield groups
$SL_2(q^d):d$, and if $G$ is orthogonal, then we use the overfield groups
$O_2^{\pm 1}(q^d):d$. The embeddings show that $d$ is a divisor of
$N_G(\langle x_d \rangle)/C_G(x_d)$. Next we note that, if $G$ is
not orthogonal and $\dim(C_N(x_d)) \geq 2$, or if $G$ is orthogonal
and $\dim(C_N(x_d)) \geq 3$, then $|C_G(x_d)|$ is divisible by $4$.

So if $ r > 3 $ is a prime divisor of $(|C_G(g)|,\phi_d(q))$, then
Lemma \ref{tuedel} implies that $G_\omega$ contains a Sylow
$r$-subgroup of $G$ and thus a conjugate of a suitable power of the
element $x_d$ above. In light of Lemma \ref{tuedel} we must have
that $|C_G(x_d)|_2 \leq 2$ which implies that $\dim(N) -
\dim([N,x_d]) = 0$ if $G$ is symplectic, $\dim(N) - \dim([N,x_d]) \leq
1$ if $G$ is linear or unitary, and  $\dim(N) - \dim([N,x_d]) \leq 2$
if $G$ is orthogonal. If $\dim(N) - \dim([N,g^i]) > 2$ for some
proper power $i$ of $g$, then the element $g^i$ is not regular,
contradicting Hypothesis \ref{betterlie}. Thus (2) is proved.

If $r = 3$ and $G_\omega$ contains an element $t$ of order $3$, then
Hypothesis \ref{betterlie} implies that $t$ is semisimple, and hence
$(3,q) = 1$. Thus if $q \equiv 1 \ \mbox{mod}~3$, then $t$ is
contained in a maximal split torus $T^+$ of $G$, and if $q \equiv \
-1 \ \mbox{mod}~3$, then $t$ is contained in a torus $T^-$ of order
$(q+1)^{\dim(N)/2}$. If $q \equiv 1 \ \mbox{mod}~3$ and $q-1 > 3$,
then Lemma \ref{tuedel} implies that $T^+ \cap G_\omega$ contains
every element of order $(q-1)/3$ of $T^+$. If $q \neq 4$, then, as
the rank of $G$ is at least $2$, some element of $T^+$ of order
$(q-1)/3$ is not regular and contained in $G_\omega$; contradicting
Hypothesis \ref{betterlie}. Similarly If $q \equiv -1 \ \mbox{mod}
~3$ and $q+1 > 3$, then Lemma \ref{tuedel} implies that $T^+ \cap
G_\omega$ contains every element of order $(q+1)/3$ of $T^+$. If $q
\neq 2$, then, as the rank of $G$ is at least $2$, some element of
$T^+$ of order $(q+1)/3$ is not regular and contained in $G_\omega$;
contradicting Hypothesis \ref{betterlie}.

If $q = 2$, then either $|T^-| \geq 27$ and $N_G(T^-)$
has a subgroup isomorphic to $\Sym_3 \wr
\Sym_{[\dim(N)/2]}$
 or $G$ is $\PSL_5(2)$. In all cases
$|N_G(\langle t \rangle)|_2 \geq 4$ for every element $t \in T^-$
and hence Lemma \ref{tuedel} implies that $|G_\omega|$ is even; a
contradiction.

If $q = 4$, then either $|T^+| \geq 27$ or the rank of $G$ is $2$.
In the former case $G_\omega$ contains elements $t$ with
$|N_G(\langle t \rangle)|_2 \geq 4$ (choose $t$ in a suitable rank $2$
subgroup of $T^+$). So Lemma \ref{tuedel} forces that $|G_\omega|$
is even, again contradicting Hypothesis \ref{betterlie}. The
classical groups of rank $2$ over the field of $4$ elements are
$\PSL_3(4)$, $\PSp_4(4)$, $\PSU_4(4)$ and $\PSU_5(4)$. The group
$\PSp_4(4)$ is a subgroup of $\PSU_4(4) = \SU_4(4)$ and $\SU_4(4)$
is isomorphic to a subgroup of $\PSU_5(4)$.  As $|\PSp_4(4)|_3 =
|\PSU_4(4)|_3 = |\PSU_5(4)|_3 = 9$ we see that every $3$-element of
$\PSU_4(4)$ and $\PSU_5(4)$ fuses to a $3$-element in $\PSp_4(4)$.

If $t$ is a $3$-element in $\PSp_4(4)$, then its centralizer in
$\PSp_4(4)$ has order divisible by $4$.

If $G = \PSL_3(4)$, then for every all $3$-elements $t \in G$ we
have that $|N_G(\langle t \rangle)| = 18$.

Now we suppose that $t \in G_\omega$ is of order $3$. Then Lemma
\ref{tuedel} implies that $|G_\omega|$ is even (see previous
paragraph) or that $G_\omega$ contains a Sylow $3$-subgroup of $G$.
The first case contradicts Hypothesis \ref{betterlie}. If, in the
second case, $T$ is a Sylow $3$-subgroup of $G$ with $T \leq
G_\omega$, then, as $|N_G(T)/T|_2 = 4$, Lemma \ref{tuedel} implies
that $|G_\omega|$ is even. This is again a contradiction. We
conclude that $3 \notin \pi(G_\omega)$ and hence (3) is proved.

Statement (4) can be deduced from Tables 5.1 and 5.2 of \cite{LSS},
whereas Statement (5) can be deduced from the tables in \cite{L}.
\end{proof}

\begin{cor} \label{notexceptional}
If $(G,\Omega)$ satisfies Hypothesis \ref{betterlie}, then $G$ is
not an
 exceptional group.
\end{cor}

\begin{proof}
Assume otherwise and let $\omega \in \Omega$. If $g \in
G_\omega^\#$, then Lemma \ref{torus}~(1) says that $T:= C_G(g)$ is a
maximal torus. Let $r$ be an odd prime and let $R$ be a Sylow
$r$-subgroup of $T$. If $r \neq 3$, then $R \leq G_\omega$ by
Lemma \ref{sylow}~(c) and thus $|N_G(R):(G_\omega \cap N_G(R)| \leq
3$ by Lemma \ref{tuedel}. It follows with Lemma \ref{torus}~(4) that
$|N_G(R)|$ is divisible by $4$ and hence $G_\omega$ has even order,
contrary to Hypothesis \ref{betterlie}. If $r = 3$, then $G_\omega$
contains a $3$-element and so Lemma \ref{torus}~(5) implies that
$|G_\omega|_2 \geq 2$, which is again a contradiction.
\end{proof}

\begin{lemma}\label{classicallarge}
If $(G,\Omega)$ satisfies Hypothesis \ref{betterlie} and $\omega \in
\Omega$, then one of the following is true:
\begin{enumerate}
\item[(1)] $G = \PSL_3(q)$ with $q$ even, $G_\omega$ is cyclic of order
$(q^2+q+1)/(3,q-1)$ (in particular of order coprime to $3$) and
$|\Omega| = (q-1)^2(q+1)q^3$. Moreover $|N_G(G_\omega):G_\omega|=3$.
\item[(2)] $G = \PSL_4(3) $, $G_\omega$ is cyclic of order $13$, and
$|\Omega| = 2^7 \cdot 3^6 \cdot 5 $.
\item[(3)] $G = \PSL_4(5)$, $G_\omega$ is cyclic of order $31$, and
$|\Omega| = 2^7 \cdot 3^2 \cdot 5^6 \cdot 13$.
\item[(4)] $G = \PSU_4(3)$, $G_\omega$ is cyclic of order $7$, and
$|\Omega| =  2^7 \cdot 3^6 \cdot 5 $.
\end{enumerate}
\end{lemma}

\begin{proof}
By (1) of Lemma \ref{torus} $C_G(g)$ is a maximal torus for every $1
\neq g \in G_\omega$. Also Lemma \ref{tuedel} and (3) of Lemma
\ref{torus} imply that $|C_G(g)|_3 \leq 3$.

If $G$ is symplectic, then the proof of (2) of Lemma \ref{torus}
showed that $C_N(g) = 0$ for every $ 1 \neq g \in G_\omega$. On the
other hand, using the fact that $g$ is contained in a subgroup of
$G$ isomorphic to $\SL_2(q^{d}):d$, where $d = \dim(N)/2$, we see
that $|N_G(\langle g \rangle)/C_G(g)| = 2d$.

By Hypothesis \ref{betterlie} we see that $|N_G(\langle g \rangle):
G_\omega \cap N_G(\langle g \rangle)|$ is even, which implies that
$G_\omega$ contains an element $h$ of order $d$ which induces a
Galois automorphism of order $d$ on $\langle g \rangle$. Now
$\dim(C_N(h)) = 2$, which means that $h$ is not regular, contrary to
Hypothesis \ref{betterlie}. So $G$ is not symplectic.

We observe that Lemma \ref{3musssein} yields that $3$ must divide
$N_G(\langle X \rangle)$ for some $X \leq G_\omega$. It follows with
Lemma \ref{tuedel} that $X$ lies inside some three point stabilizer
$H$.

Now if for all $g \in H$ and all $1 \neq h \in \langle g \rangle$ we
have that $N_G(\langle h \rangle)  \le H$, then Lemma \ref{charfrob}
implies that $G$ is a Frobenius group, contrary to our main
hypothesis. Therefore we find $1 \neq g \in H$ such that
$|N_G(\langle g \rangle ): N_H(\langle g \rangle)| = 3 $.

If $G$ is linear or unitary, then the elements $g \in G_\omega^\#$
satisfy $\dim(N) - \dim([N,g^i]) \leq 1$. Thus the order of every
such $g$ is a divisor of $q^{\dim(N)}-1$ or $(q^{\dim(N) -1}-1)$.
Now using the fact that some nontrivial subgroup of $H$ has to have
a normalizer whose order is divisible by $3$ implies that either
$\dim(N) \ \equiv \ 0 \ \mbox{mod}~3$ or $\dim(N) -1 \ \equiv \ 0 \
\mbox{mod}~3$. On the other hand Lemma \ref{torus} shows that if
$o(g)$ is a divisor of $q^{\dim(N)}+1$, of $q^{\dim(N)}-1$ or of
$(q^{\dim(N) -1} \pm 1)$, then $|N_G(\langle g \rangle)/C_G(g)|$ is
divisible by $\dim(N)$ or $\dim(N)-1$, respectively.

An element $h$ of order $\dim(N)$ or $\dim(N)-1$ and divisible by
$3$ that lies in $N_G(\langle g \rangle) \setminus C_G(g)$ has the
property that $\dim(C_N (h^3)) \geq 3$. Therefore it does not lie in
$G_\omega$ by Hypothesis \ref{betterlie}. Thus $|N_G(\langle g
\rangle): G_\omega \cap N_G(\langle g \rangle)| \geq \dim(N) -1$
which implies that $\dim(N) \leq 4$ and $o(g)$ is a divisor of $q^3
+ 1$ or $q^3-1$. Then also $C_G(g) \leq G_\omega$.

If $\dim(N) = 4$ and $G$ is linear, then $|C_G(g)| =
(q^3-1)/(4,q-1)$. If $q \neq 3,5$, then $C_G(g)$ contains elements
of order dividing $(q-1)$ whose centralizer in $G$ contains a
subgroup isomorphic to $\SL_3(q)$. But then Lemma \ref{tuedel}
implies that $G_\omega$ contains such a subgroup, which is a
contradiction. When $q = 3$ or $5$, then $C_G(g)$ is cyclic of order
$(q^3-1)/(4,q-1) = q^2 + q + 1$ and $N_G(\langle g \rangle)/C_G(g)$
is cyclic of order $3$. Thus we obtain examples (2) and (3) from the lemma.

If $\dim(N) = 3$ and $G$ is linear, then $q$ is even, $|C_G(g)| =
(q^2+q+1)/(3,q-1)$ and $N_G(\langle g \rangle)/C_G(g)$ is cyclic of
order $3$.

As $3$ divides $(q^2-1)$, but not $|C_G(g)|$, this implies that
$G_\omega = C_G(g)$. Moreover $(q-1,q^3-1) = 3(q-1)$ and so we see
that $C_G(y) = C_G(g)$ for all $y \in C_G(g)^\#$.

Thus $|F_\Omega(y)| = |N_G(\< y \>):G_\omega| =
|N_G(G_\omega):G_\omega| = 3$ as in (1).

If $G$ is unitary, then $\dim(N) \geq 4$ and $q > 2$ by Hypothesis
\ref{betterlie}, and hence $\dim(N) = 4$. In this case $|C_G(g)| =
(q^3+1)/(4,q+1)$. If $q > 3$, then $C_G(g)$ contains elements of
order dividing $(q+1)$ whose centralizer in $G$ has a subgroup
isomorphic to $\SU_3(q)$. But then Lemma \ref{tuedel} implies that
$G_\omega$ contains a subgroup isomorphic to $\SU_3(q)$, which is a
contradiction. Note that $\PSU_4(2) \cong \PSp_4(3)$ does not give
any examples by Lemma \ref{Sp43}. Finally if $G = \PSU_4(3)$, then
$|C_G(g)| = (3^3+1)/(4,3+1) = 7$ and $|N_G(\< g\>)/C_G(g)| =3$,
which yields example (4) in the conclusion of our lemma.

If $G$ is orthogonal, then $\dim(N) \geq 7$ and $\dim([N,g])$ is
even. Therefore $|N_G(\< g\>)/C_G(g)| \geq 6$, which implies that
$G_\omega = C_G(g)$. This means that $G_\omega$ contains an
involution, contradicting our hypothesis that $|G_\omega|$ is odd.
\end{proof}

\subsection{Summary}

\begin{thm}\label{liemain}
Suppose that $G$ is simple and of Lie type and that $G$ is not
isomorphic to an alternating group. Suppose further that
$(G,\Omega)$ satisfies Hypothesis \ref{3fix}. Then one of the
following is true:
\begin{enumerate}
\item[(1)] $G = \PSL_3(q)$, $G_\omega$ is cyclic of order $(q^2+q+1)/(3,q-1)$,
and $|\Omega| = (q-1)^2(q+1)q^3$.
\item[(2)] $G = \PSL_4(3) $, $G_\omega$ is cyclic of order $13$, and
$|\Omega| = 2^7 \cdot 3^6 \cdot 5 $.
\item[(3)] $G = \PSL_4(5)$, $G_\omega$ is cyclic of order $31$, and
$|\Omega| = 2^7 \cdot 3^2 \cdot 5^6 \cdot 13$.
\item[(4)] $G = \PSU_3(q)$ with $q \geq 3$, $G_\omega$ is cyclic of order
$(q^2-q+1)/(3,q+1)$ and $|\Omega| = (q-1)(q+1)^3q^3$.
\item[(5)] $G = \PSU_4(3)$, $G_\omega$ is cyclic of order $7$, and
$|\Omega| =  2^7 \cdot 3^6 \cdot 5$.
\item[(6)] $G = \PSL_2(7) \cong \PSL_3(2)$ with $|\Omega| = 7$ and
$G_\omega \cong \Sym_4$.
\item[(7)] $G = \PSL_2(11)$ with $|\Omega| = 11$ and $G_\omega \cong \Alt_5$.
\end{enumerate}
We note that in (1) and (4) the point stabilizers have order coprime to $6$.
\end{thm}

\begin{proof}
The groups with strongly $2$-embedded subgroups were considered in
Lemma \ref{PSL2e}, Lemma \ref{Szq} and Lemma \ref{PSU3e}. The only
examples arising here are the groups  $\PSU_3(q)$ where $q$ is even,
as described in (4).

The groups with dihedral or semidihedral Sylow $2$-subgroups were
considered in Lemma \ref{PSL2o}, Lemma \ref{PSU3o} and Lemma
\ref{PSL3o}. The examples arising here are the groups  $\PSU_3(q)$
with $q$ odd, which are accounted for in (4), the groups $\PGL_3(q)$
with $q$ odd, which appear in (1), and the groups $\PSL_2(7)$ and $
\PSL_2(11) $ which are listed in (6) and (7).

The groups for which the normalizer of a Sylow $2$-subgroup is not
strongly embedded and where the Sylow $2$-subgroups are neither
semidihedral nor dihedral satisfy Hypothesis \ref{newlie}. In fact
all but $\PSp_4(3)$ and $^2G_2(q)$ satisfy Hypothesis
\ref{betterlie}. Lemma \ref{Sp43} shows that the group $\PSp_4(3)$
does not give any example and Lemma \ref{2G2} shows the same for the
groups $^2G_2(q)$.

The exceptional groups of Lie type which satisfy Hypothesis
\ref{betterlie} do not lead to examples and this is the content of
Corollary \ref{notexceptional}. The classical groups of Lie type
which satisfy Hypothesis \ref{betterlie} are treated in Lemma
\ref{classicallarge} and here the examples involving $\PSL_3(q)$
with $q$ even, $\PSL_4(3), ~\PSL_4(5), ~\PSU_4(3)$ arise. These are
accounted for in (1), (2), (3) and (5), respectively.
\end{proof}

For convenience we remind the reader that $\PSL_2(7) \cong
\PSL_3(2)$ gives rise to two examples which are listed in (1) and
(6) above. The next results deal with almost simple groups that play
a role for Theorems \ref{almostsimple3fp} and \ref{main}.

Before analyzing the almost simple groups with socle $\PSL_3(q)$ and
$\PSU_3(q)$, we need a preparatory lemma.

\begin{lemma} \label{autprep}
Suppose that $p$ is a prime and let $a \in \N$ and $q := p^a > 4$,
and let $E:=\PSL_3(q)$. Then $|\Out(E)|=2a \cdot (q-1,3)$.

Now suppose that $G$ is a group such that $E < G \leq \au(E)$ and
that $(G,\Omega)$ satisfies Hypothesis \ref{3fix}. Let $\omega \in
\Omega$ and suppose that $E_\omega$ is cyclic of order
$q^2+q+1/(q-1,3)$. Then the following are true:

\begin{enumerate}
\item[(1)]
$N_G(E_\omega)/N_E(E_\omega) \cong G/E$.
\item[(2)]
If $q > 4$, then $(q-1,3) = 3$, $G/E$ is cyclic of order $3$, and no
element of $G \setminus E$ induces a field automorphism on $E$.
\end{enumerate}
\end{lemma}

\begin{proof} The order of the outer automorphism group of $E$
is well known and is as claimed. The outer automorphisms are
diagonal, field or graph automorphisms and their products. All of
this can be found in Theorem 2.5.12 of \cite{GLS3}.

Now (1) follows from a Frattini argument, using the fact that
$G$ acts transitively on the set of point stabilizers in $E$. \\

We know from Lemmas \ref{PSL3o} and \ref{classicallarge} that
$|N_E(E_\omega)| = 3 \cdot |E_\omega|$ and that $(|E_\omega|,3) =
1$. Thus (1) implies that $|G_\omega:E_\omega| = |G/E|$.

To prove (2) we suppose to the contrary that $1 \neq
|N_G(E_\omega)/N_E(E_\omega)| =: b$. Then Lemma \ref{tuedel} implies
that $G_\omega$ contains a subgroup of order $b$. Now if $(b,3) =
1$, then all elements $h \in N_G(E_\omega) \setminus N_E(E_\omega)$
of prime order dividing $b$ are either graph, field or graph-field
automorphisms. Thus, as $q > 4$, it follows that $C_E(h)' \cong
~\PSO_3(q), ~\PSL_3(q_0)$ or $\PSU_3(q_0)$, where $q_0$ divides $q$.
As no proper subgroup of $E$ contains both $E_\omega$ and $C_E(h)'$,
we see that $N_G(E_\omega)/E_\omega$ is a $3$-group.

Next we note that if $(3,q-1) =1$, then $\PGL_3(q) \cong \PSL_3(q) =
E$ and hence every element $t$ of order $3$ in $G \setminus E$ is a
field automorphism such that $C_E(t)' \cong \PSL_3(q_0)$. Now (1)
forces a conjugate of $t$ into $G_\omega$. However, as no proper
subgroup of $E$ can contain $E_\omega$ and $C_E(t)'$ we see that $3
= (q-1,3)$.

Thus $N_G(E_\omega)/E_\omega$ is a $3$-group and $(3,q-1) =3$ and, with $l$ denoting the highest power of $3$ dividing $a$ (from our hypothesis), we see that $G/E$ is a $3$-group of order at most $3 \cdot l.$

In case $G/E$ contains field automorphisms of order $3$, then (1)
implies that $G_\omega$ contains a field automorphism $t$ of order
$3$ such that $C_E(t) \cong \PSL_3(q_0)$. As before Lemma
\ref{tuedel} implies $E \leq G_\omega$, which is impossible. The
fact that no element of $G \setminus E$ is allowed to induce a field
automorphism of $E$ implies that $|G/E| = 3$, our claim.
\end{proof}

\begin{lemma}\label{AutPSL3}
Suppose that $G$ is almost simple and not simple and that $E = F^*(G) \cong \PSL_3(q)$.
If $(G,\Omega)$ satisfies Hypothesis \ref{3fix},
then $(3,q-1) = 3$, $G = \PGL_3(q)$ and $G_\omega$ is cyclic of order $(q^3-1)/(q-1)$.
\end{lemma}

\begin{proof}
If $F^*(G) =\PSL_3(q)$ with $q \leq 4$, then the table of marks for
the almost simple groups of this type are in \cite{GAP}. Inspection
of these tables yields exactly our claimed example; i.e. $\PGL_3(4)$
acting on the cosets of a cyclic group of order $21$.

So without loss we may assume that $q > 4$.
Let $\omega \in \Omega$.
First we note that $E_\omega$ is cyclic of order $(q^2+q+1)/(3,q-1)$
 by Theorem \ref{liemain}~(1).
Moreover by Lemma
\ref{autprep} we know that $3 = (q-1,3)$ and either $G \cong
\PGL_3(q)$ or $q = q_0^3$ and $G \cong \PGL_3^*(q) = \< E, d \>$
where $d$ induces a diagonal-field automorphism on $E$.

In the latter case we see, by direct computation, that any $g \in
\PGL_3^*(q) \setminus E$ has order divisible by $9$. Thus if $g \in
G_\omega \setminus E_\omega$, then $g$ has order $9$ which implies
that $3$ divides $|E_\omega|$, contradicting Theorem \ref{liemain}.
Hence $G \cong
\PGL_3(q)$.
We note
that $G_\omega \le N_G(E_\omega)$, but that $G_\omega \nleq E$ by Lemma \ref{autprep}~(1).\\

We note that a Singer cycle in $\GL_3(q)$ has order $(q^3-1)$ and maps
via the natural projection to a cyclic subgroup $C$ of order
$(q^3-1)/(q-1)=q^2+q+1$ of
$\PGL_3(q)$. It follows from the subgroup structure of $\PGL_3(q)$ that
$C \cap E$ is conjugate to $E_\omega$. (They have the same order and are both cyclic.)
So we may suppose that $E_\omega \leq C \leq N_G(E_\omega) = \< C, t\>$ where
$t \in N_E(E_\omega)$
is an element of order $3$.
We note that $N_G(E_\omega)/E_\omega$ is elementary abelian of order $9$ and now we let
$d \in C$ be of order $3$ and such that
$\< d,t\>$ is a Sylow
$3$-subgroup of $N_G(E_\omega)$. In particular $C=\<E_\omega,d\>$. \\

Now we see four possibilities for $G_\omega$ ($\leq N_G(E_\omega))$ because $\<d,t\>$ has four subgroups of order $3$.

The first possibility is that $G_\omega=\<E_\omega,t\>$. But this is impossible because
$t \in E$ and $G_\omega \nleq E$.
Now we assume that $G_\omega \in \{\<E_\omega, dt \>, \< E_\omega, d^{-1}t \>\}$.

Let $h \in \{dt, d^{-1}t\}$ (depending on $G_\omega$) and choose
$g \in K:=\GL_3(q)$ to be a
$3$-element that projects onto $h$. Then
$|C_{K}(g)| \geq (q-1)^2$, which implies that $|C_G(h)| \geq (q-1) >
3$ (because $q > 4$). Now $h \in G \setminus E$ and so $N_G(\< h \>)
= C_G(h)$. Hence $|N_G(\< h \> )|= |C_G(\< h \>)| = 3|C_G(h)| \geq
3(q-1) > 9$ and thus Lemma \ref{tuedel} implies that $N_G(\< h \>)
\leq N_G(C)$ and thus that $C_G(h) = C_H(h)$. On the other hand
$|C_C(h)| = 3$, as $t$ acts fixed point freely on $E_\omega$ and
thus $9 = |C_H(h)| < |C_G(h)| $, which is a contradiction.

Now there is only one possibility left, namely that $G_\omega=\<E_\omega,d\>=C$.\\

Finally we observe that the possibility $G_\omega = C$ leads to an
example. To see this it suffices to observe that $N_G(\< c \>) \leq
N_G(C)$ for all $1 \neq c \in C \setminus E_\omega$. The latter is
clear as $C_G(c) \leq C_G(d) = C$.
\end{proof}

\begin{lemma}\label{AutPSU3}
Suppose that $q$ is a prime power, $q \neq 2$, and that $G$ is
almost simple, but not simple, with $F^*(G) \cong \PSU_3(q)$. If
$(G,\Omega)$ satisfies Hypothesis \ref{3fix}, then $(3,q+1) = 3$, $G
= \PGU_3(q)$ and $G_\omega$ is cyclic of order $(q^3+1)/(q+1)$.
\end{lemma}

\begin{proof} Theorem \ref{liemain} in combination with the main
theorem of \cite{MW} implies that the only possible action for
$F^*(G)$ is the action on the set of cosets of a cyclic group $C$ of
order $(q^2-q+1)/(3,q+1)$.
By observing that $\GU_3(q)$ lies naturally in $\GL_3(q^2)$ such that
the group $C$ lies naturally in the cyclic group $E_\omega \leq \PSL_3(q^2)$
of order $(q^4+q^2+1)/(3,q^2-1)$ from Lemma \ref{AutPSL3} above, we may
use the computations from Lemma \ref{autprep} and Lemma \ref{AutPSL3}
to establish our claim. We omit the details.
\end{proof}

\begin{lemma} \label{autsingles} Suppose that
$(G,\Omega)$ satisfies Hypothesis \ref{3fix} and that $G$ is almost
simple such that $F^*(G)$ is one of $ \PSL_4(3), ~\PSL_4(5)$ or
$\PSU_4(3)$. Then $G$ is simple.
\end{lemma}

\begin{proof}
Let $\omega \in \Omega$ and suppose that $F^*(G)$ is one of
$\PSL_4(3), ~\PSL_4(5)$ or $\PSU_4(3)$. Then $P:=G_\omega \cap
F^*(G)$ is a Sylow $13$-, $7$- or $31$-subgroup of $G$,
respectively. Now we note that $P<N_{F^*(G)}(P)$ by Theorem \ref{liemain}, but also 
$G=F^*(G) \cdot N_G(P)$ by Frattini. Hence Lemma \ref{tuedel} forces an involution $t \in N_G(P)$ into $G_\omega$.
However the structure of $C_{F^*(G)}(t)$ and Lemma \ref{tuedel} then imply
that $F^*(G) \cap G_\omega \neq P$, which is a contradiction.
\end{proof}

\section{The sporadic simple groups}

In this section we adapt the notation in the ATLAS (\cite{ATLAS})
for the names of the sporadic groups.

\begin{lemma}\label{M11M12}
Suppose that $G$ is $M_{11}$ or $M_{12}$ and that ($G,\Omega$)
satisfies Hypothesis \ref{3fix}. Then $G=M_{11}$ in its $4$-fold
sharply transitive action on $11$ points.
\end{lemma}

\begin{proof}
Let $\alpha \in \Omega$ and $H:=G_\alpha$. Let $x \in H$ and
$X:=\langle x \rangle$. For maximal subgroups of $G$ and information
about local subgroups we refer to Tables 5.3a and 5.3b in
\cite{GLS3}.

First assume that $x$ has order $11$.
Then $N_G(X)$ has order $11 \cdot 5$ and therefore
Lemma \ref{tuedel} yields that $H$ contains a subgroup $Y$ of order $5$.
In both cases, $|N_G(Y)|$ is divisible by $4$ and hence $H$ has even order, again by Lemma \ref{tuedel}.
Then let $t \in H$ be an involution.
Lemma \ref{tuedel} implies that $H$ contains a subgroup of index at most $3$ of $C_G(t)$.
As $|H|$ is also divisible by $11$ and by $5$, the lists of maximal subgroups yield that $H=G$.
This is impossible.

If $x$ has order $5$, then $N_G(X)$ has order divisible by $4$ and
hence $H$ contains a subgroup of index at most $3$ of an involution
centralizer (applying Lemma \ref{tuedel} twice). This is possible if
$G=M_{11}$ and $\Omega$ has $11$ elements, and we already know that
this is in fact an example for Hypothesis \ref{3fix}. In $M_{12}$,
we see that $H$ lies in the centralizer of an involution from class
$2A$ and hence $H$ contains a $3$-element. Lemma \ref{tuedel}
implies that $9$ divides $|H|$, but this is false.

So from now on we consider the case where $H$ is a $\{2,3\}$-group.\\

Let us assume that $x$ is an involution and that
$|\FO(x)| \in \{1,3\}$. Then all involutions have an odd number of fixed points and hence
Lemma \ref{sylow} (or \ref{tuedel}~(a)) yields that $H$ has odd index. In $M_{12}$ we
immediately have $3 \in \pi(H)$ via $C_G(x)$ and Lemma \ref{tuedel}. In $M_{11}$ we look at a
fours group in $H$ and apply Lemma \ref{tuedel} to it in order to see that $3 \in \pi(H)$.
Let $Y \le H$ be a subgroup of order $3$.

If $G=M_{11}$, looking at the list of maximal subgroups, we see that
$H$ does not contain a Sylow $3$-subgroup of $G$ in this case. So we
may suppose that $|\FO(Y)|=3$ by Lemma \ref{tuedel}~(a) and (b). It
follows that $H=C_G(x)$. Let $a \in H$ be an element of order $8$.
As $|\Omega|=165$, we see that $x$ has either one fixed point, one
orbit of length $4$ and regular orbits or three fixed points, one
orbit of length 2 and regular orbits on $\Omega$. In both cases
$a^4$ is an involution that has too many fixed points.

If $G$ is $M_{12}$, then $H$ contains a full involution centralizer.
This implies that $5 \in \pi(H)$, which is a contradiction.

Suppose that $o(x)=3$. Then $N_G(X)$ has order divisible by $4$ and
hence Lemma \ref{tuedel} yields that $H$ has even order. Let $t \in
H$ be an involution. We already treated the case where some
involution in $H$ has one or three fixed points, so $|\FO(t)|=2$ and
in particular $|\Omega|$ is even. Lemma \ref{tuedel}~(b) yields that
$H$ contains an index two subgroup of an involution centralizer,
which in the case $M_{12}$ implies that $H$ contains involutions
from all conjugacy classes (see Lemma \ref{2isttoll}). In particular
$H$ contains a Sylow $2$-subgroup of $G$, contrary to the fact that
$|\Omega|$ is even. If $G=M_{11}$, then $H$ contains subgroups of
structure $\SL_2(3)$ and $\sym_3 \times 2$, which is also
impossible.

This finishes the proof.
\end{proof}

The remaining sporadic groups do not have dihedral or semidihedral
Sylow $2$-subgroups. This makes Lemma \ref{2isttoll} very useful
again.

\begin{lemma}\label{tollspor}
Suppose that Hypothesis \ref{3fix} holds and that $G$ is a sporadic
group, but not $M_{11}$. Let $\alpha \in \Omega$. Then $G_\alpha$
contains a Sylow $2$-subgroup of $G$ or it has odd order. In the
second case there exists no prime $p \in G_\alpha$ such that $p
\vdash 2$.
\end{lemma}

\begin{proof}
This is a combination of Lemmas \ref{2isttoll} and \ref{key}.
\end{proof}

\begin{lemma}\label{Mathieu}
Suppose that $G$ is $M_{22}$, $M_{23}$ or $M_{24}$ and that $\Omega$
is a set such that ($G,\Omega$) satisfies Hypothesis \ref{3fix}.
Then $G=M_{22}$ and the action of $G$ on $\Omega$ is as the action
of $G$ on the set of cosets of a subgroup of order $7$.
\end{lemma}

\begin{proof}
Let $\alpha \in \Omega$, let $x \in H:=G_\alpha$ and $X:=\langle x
\rangle$. We may suppose that $x$ has prime order $p$. For maximal
subgroups of $G$ and information about local subgroups we refer to
Tables 5.3c-e in \cite{GLS3}.

We first suppose that $H$ has odd order, in particular $p$ is odd
and $p \not \vdash 2$. In all groups considered here, $11 \vdash 5$
and $5 \vdash 2$, so $p$ is neither $11$ nor $5$. Moreover $23
\vdash 11$ and hence $p \neq 23$. If $p=7$, then this either leads
to $M_{22}$ and the example that is stated in the lemma, or, in the
larger groups, we have that $7 \vdash 3$. But also $3 \vdash 2$, so
this leads to a contradiction.

Lemma \ref{tollspor} leaves the case where $H$ contains a Sylow
$2$-subgroup of $G$. Looking at centralizers of involutions (and in
$M_{22}$, also at the normalizer of an elementary abelian subgroup
of order $8$), we see that $3 \in \pi(H)$ by Lemma \ref{tuedel}.

If $G=M_{22}$, then $H$ lies in a maximal subgroup of structure
$2^4:\A_6$ or $2^4:\Sym_5$, so by Lemma \ref{tuedel} it is equal to
one of these groups. But this does not agree with Lemma
\ref{tuedel}.

If $G$ is $M_{23}$ or $M_{24}$, then $H$ contains a full involution
centralizer. In $M_{23}$ this means that $H$ is a maximal subgroup
of structure $2^4:\Alt_7$. Then by congruence modulo $3$, all
$3$-elements must have a unique fixed point and Lemma \ref{tuedel}
forces $H$ to contain  a subgroup of structure $(3 \times \Alt_5)
\cdot 2$. This is a contradiction. In $M_{24}$ we see that $H$
contains a Sylow $2$-subgroup of $G$, hence $|\Omega|$ is odd. It is
also coprime to $5$ and $7$, and in the only remaining possible case
it follows that elements of order $5$ in $H$ have too many fixed
points.
\end{proof}

\begin{lemma}\label{Janko}
Suppose that $G$ is a Janko Group. Then
there is no set $\Omega$ such that ($G,\Omega$) satisfies Hypothesis \ref{3fix}.
\end{lemma}

\begin{proof}
Assume that
$\Omega$ is such a set, let $\alpha \in \Omega$ and $H:=G_{\alpha}$.
For information about local subgroups of $G$ we refer to Tables 5.3f-i in \cite{GLS3} whereas
we use the lists of
maximal subgroups of $G$ from Tables 5.4 and 5.11 in \cite{W}.

First suppose that $H$ has odd order and let $x \in H$ be of prime
order $p$. We note that $3 \vdash 2$ and $5 \vdash 2$, so $p \ge 7$.
Then the tables yield that also $p \neq 7,19$. Moreover If $11
\vdash 5$, $17 \vdash 2$, $23 \vdash 11$, $29 \vdash 7$ and $43
\vdash 7$. The only remaining primes are $31$ and $37$, but they are
also impossible because $31 \vdash 5$ and $37 \vdash 2$. Hence this
case does not occur at all.

With Lemma \ref{tollspor} we know that $H$ contains involutions from
all conjugacy classes. In particular $3,5 \in \pi(H)$ whence, by
Lemma \ref{sylow}, we see a Sylow $5$-subgroup of $G$ in $H$.

If $G=J_1$, then Lemma \ref{tuedel} yields that $H$ contains
a subgroup of shape $3 \times D_{10}$ or $\Sym_3 \times 5$ and an $\Alt_5$.
There is no maximal subgroup that could contain $H$ now.

If $G=J_2$, then $H$ is contained in a maximal subgroup of structure $\A_5 \times D_{10}$ or $5^2:D_{12}$
(by its index in $G$). Both cases are impossible because $9$ divides $|H|$ by Lemma \ref{tuedel}.

If $G=J_3$, then there is only one type of maximal subgroup that contains a Sylow $5$-subgroup and a
subgroup of order $3^3$, and it has structure $(3 \times \A_6):2$. But its order is only divisible by
$2^4$ and not by $2^6$, so it cannot contain $H$.

In the last case $G=J_4$, we see that the centralizer of an involution involves the group $M_{22}$. Hence
Lemma \ref{tuedel} yields that $|H|$ is divisible not only by $2,3$ and $5$, but also by $7$ and $11$,
hence by $11^3$ (using Lemma \ref{sylow}~(c)).
There are only two types of maximal subgroups that have order divisible by $11^3$, and in both cases their
order is not divisible by $7$.
This is a contradiction.
\end{proof}

\begin{lemma}\label{Co}
Suppose that $G$ is a Conway Group. Then there is no set $\Omega$
such that ($G,\Omega$) satisfies Hypothesis \ref{3fix}.
\end{lemma}

\begin{proof}
Assume otherwise, let
$\Omega$ denote such a set, let $\alpha \in \Omega$ and $H:=G_\alpha$.
For information about local subgroups of $G$ we refer to Tables 5.3j-l in \cite{GLS3} and
for lists of
maximal subgroups of $G$ and their indices we use \cite{ATLAS}.

The tables yield that for all prime divisors $p$ of $G$, we have
that $p \rightarrow 2$. Hence it is impossible that $H$ has odd
order, by Lemma \ref{key}. Lemma \ref{tollspor} implies that $H$
contains involutions from all conjugacy classes. This yields that
$3,5 \in \pi(H)$. In particular $H$ contains a full Sylow
$5$-subgroup of $G$ by Lemma \ref{sylow}~(c). Inspection of the
lists of maximal subgroups of $G$ shows that all maximal subgroups
have index divisible by $5$ or by $2$, which is a contradiction.
\end{proof}

We use the idea from the previous lemma to find a general
approach for almost all the remaining sporadic groups.

\begin{thm}\label{sporhelp}
Suppose that $G$ is one of the following sporadic simple groups:
$HS, McL, Suz, He, Ly, Ru, O'N, Fi_{22}, Fi_{23}, F'_{24}, HN, Th,
BM$. Then there is no set $\Omega$ such that $(G,\Omega$) satisfies
Hypothesis \ref{3fix}.
\end{thm}

\begin{proof}
Assume otherwise and let $\Omega$ be such that ($G,\Omega$)
satisfies Hypothesis \ref{3fix}. Let $\alpha \in \Omega$ and
$H:=G_{\alpha}$. For information about local subgroups of $G$ we
refer to Tables 5.3m-y in \cite{GLS3} and for lists of maximal
subgroups of $G$ and their indices we use \cite{ATLAS} unless stated
otherwise.

\begin{enumerate}
\item[(1)]
$2,3,5 \in \pi(H)$.

\begin{proof}
In all groups we see that for all odd $p \in \pi(G)$, we have that
$p \rightarrow 2$ and hence $H$ has even order. It contains
involutions from all conjugacy classes by Lemma \ref{tollspor} and
so we see that also $3,5 \in \pi(H)$ by Lemma \ref{tuedel}.
\end{proof}

\item[(2)]
$H$ is contained in a maximal subgroup of index that is odd and
coprime to $5$.

\begin{proof}
We know from (1) and from Lemma \ref{sylow}~(c) that $H$ contains a
Sylow $5$-subgroup of $G$. Moreover $H$ contains a Sylow
$2$-subgroup of $G$ by (1) and Lemma \ref{tollspor}. The same holds
for a maximal subgroup containing $H$ and hence the statement about
the index follows.
\end{proof}
\end{enumerate}

We inspect the lists of maximal subgroups of the groups and in particular their indices.
In most cases, this already contradicts (2).
For lists of maximal subgroups of the Fischer sporadic simple groups we refer to Table 5.5 in \cite{W}
(particularly because there is a mistake in the list of subgroups of $\Fi_{23}$ in \cite{ATLAS}).
For BM, we refer to Table 5.7 in \cite{W}.

For Th, there is one maximal subgroup missing in the ATLAS, namely $\PSL_3(3)$ (see Table 5.8 in \cite{W}).
Its index is divisible by $2^{11}$ and by $5^3$, so this possibility contradicts (2).
For $\Fi_{24}'$, we also note that the maximal subgroups of structure $\PSU_3(3):2$
and $\PGL_2(13)$ cannot contain $H$ because of (2).

However, there are a few exceptions.

If $G=O'N$, then $H$ could be contained in a maximal subgroup of structure
$4\dot{~} \PSL_3(4):2$.
Then $H$ contains subgroups of order $5$ and $7$, so by Lemma \ref{sylow}~(c) it follows that $H$
contains a Sylow $7$-subgroup of $G$. This has order $7^3$, which is impossible.

If $G=\Fi_{23}$, then $H$ could be contained in an involution centralizer of structure $2\Fi_{22}$.
In particular $H$ contains a subgroup of order $3^9$ and hence a $3$-central element of $G$. Lemma
\ref{tuedel} implies that $3^{12}$ divides $|H|$, but this is false.
\end{proof}

\begin{lemma}\label{M}
There is no set $\Omega$ such that ($M,\Omega$) satisfies Hypothesis
\ref{3fix}.
\end{lemma}

\begin{proof}
Assume otherwise and let $G$ denote the Monster sporadic group $M$.
Let $\Omega$ be such that ($G,\Omega$) satisfies Hypothesis
\ref{3fix}, let $\alpha \in \Omega$ and let $H:=G_{\alpha}$. We
refer to Table 5.3z in \cite{GLS3} for information about local
subgroups and to Table 5.6 in \cite{W} for the list of known maximal
subgroups of $G$.

First we show that $H$ has even order. This follows easily because,
if $p$ is any odd prime divisor of $G$, then inspection of the
tables shows that $p \rightarrow 2$. Then we use Lemma \ref{key}. It
follows from Lemma \ref{tollspor} that $H$ contains involutions from
both conjugacy classes, so looking at the involution centralizers in
Table 5.3z in \cite{GLS3}, Lemma \ref{tuedel} tells us that $H$
contains a subgroup isomorphic to $BM$ and to $Co_1$. Checking the
list of known maximal subgroups of $G$, we already see that this
does not occur.

On page 258 in \cite{W} it is noted (quoting work of Holmes and
Wilson) that if $U$ is any other maximal subgroup of $G$, then there
exists a group $E$ isomorphic to one of $\PSL_2(13)$, $\PSU_3(4)$,
$\PSU_3(8)$, $\Sz(8)$, $\PSL_2(8)$, $\PSL_2(16)$ or $\PSL_2(27)$
such that $E \le U \le \au(E)$. Checking the possibilities for $U$
with these constraints, we see that $U$ does not have a subgroup
isomorphic to $BM$ or to $Co_1$ and therefore $H$ cannot be
contained in a maximal subgroup $U$ of $G$ of this kind.
\end{proof}

All results of this section together yield the following:

\begin{thm}\label{mainspor}
Suppose that $G$ is a sporadic simple group and that $\Omega$ is
such that $(G, \Omega)$ satisfies Hypothesis \ref{3fix}. Then
$G=M_{11}$ and $|\Omega|=11$ or $G=M_{22}$ and $|\Omega|=2^7 \cdot
3^2 \cdot 5 \cdot 11$.
\end{thm}

\section{Proof of the main results}

\begin{lemma}\label{o3}
Suppose that $N$ is an elementary abelian normal subgroup of $G$ and
that $H$ is a t.i. subgroup of $G$ of order coprime to $6$. Suppose
further that $|N_G(X):N_H(X)| = 3$ for all subgroups $1 \neq X \leq
H$ and that $|C_N(H)| = 3$. Then $H$ has a normal complement $K$ in
$G$.
\end{lemma}

\begin{proof}
Our hypotheses imply that $N$ is a $3$-group, which means that $H$
acts coprimely on $N$ and therefore $N=C_N(H) \times [N,H]$.
 and
that $N_G(H) = C_N(H) \times H$. Moreover $[N,H]$ is an
$H$-invariant subgroup of $N$, in particular $[N,H]H$ is a subgroup
of $NH$. Now let $h \in H^\#$ and $x \in [N,H]$ be such that
$x^h=x$. Then $h \in H \cap H^x$, so $H= H^x$ because $H$ is a t.i.
subgroup. This means that $[H,x] \le H \cap N=1$ and therefore $x
\in C_N(H)$. This forces $x=1$ and we deduce that $[N,H]H$ is a
Frobenius group with complement $H$. As $|H|$ is odd, the Sylow
subgroups of $H$ are cyclic and in particular $H$ is metacyclic (see
8.18 in \cite{Hupp}). Also we see that $Z(NH) = C_N(H)$.

Let $n \in \N$ and let $p_1,..,p_n$ be pair-wise distinct prime
numbers such that $\pi(H)=\{p_1,...,p_n\}$ and $p_1 < \cdots < p_n$.
Let $P_1 \in \sy_{p_1}(H)$. We recall that $H$ is a
$\{2,3\}'$-group, so we know that $p_1 \ge 5$ and hence $P_1 \in
\sy_{p_1}(G)$ by Lemma \ref{sylow}~(c).

As $P_1$ is cyclic and $p_1$ is the smallest element in $\pi(H)$ we
see that $|\au(P)|_{p_1'} = (p_1-1) < p_2,...,p_n$. This means that
$N_H(P_1) = P_1$. Thus $|N_G(P_1):N_H(P_1)| = 3$ by hypothesis and
it follows that $N_G(P_1) = C_N(H) \times N_H(P_1) = C_N(H) \times
P_1$. Burnside's $p$-complement theorem implies that $P_1$ has a
normal $p_1$-complement $M_1$ in $G$. We recall that $p_1 \ge 5$ and
hence $N \le M$. Moreover $H_1:=H \cap M_1$ is characteristic in $H$
and so $N_G(H_1) = C_N(H) \times H$.

We show that $M_1, H_1$ and $N$ satisfy the hypotheses of the lemma
instead of $G,H$ and $N$. Of course $N$ is an elementary abelian
normal subgroup of $M_1$ and $H_1$ is a $\{2,3\}'$-group. Let $g \in
M_1$ be such that $H_1 \cap H_1^g \neq 1$. Then $1 \neq H \cap M_1
\cap H^g$, in particular $H \cap H^g \neq 1$ and thus $H=H^g$.
Therefore $H_1 \cap H_1^g=H \cap M_1 \cap H^g=H \cap M_1=H_1$, which
means that $H_1$ is a t.i. subgroup of $M_1$. If $1 \neq Y \le H_1$,
then $N_G(Y)=N_H(Y) \times C_N(H)$ by hypothesis and hence
$N_{M_1}(Y)=N_{H_1}(Y) \times C_N(H)$. In particular
$|N_{M_1}(Y):N_{H_1}(Y)| = 3$.

We continue in this way: $p_2 \ge 7$ and hence $H_1$ contains a
Sylow $p_2$-subgroup $P_2$ of $G$, hence of $M_1$ (by Lemma
\ref{sylow}~(c). Arguing for $M_1$, $H_1$ and $P_2$ as for $G$, $H$
and $P_1$ before, we find a normal $p_2$-complement $M_2$ in $M_1$.
Then $M_2$ is characteristic in $G$, in fact
$M_2=O_{\{p_1,p_2\}'}(G)$ and $M_2$ contains $N$, so we may repeat
these arguments until we reach the largest prime divisor of $|H|$.
This way we find a normal complement for $H$ in $G$, namely
$O_{\pi(H)'}(G)$.
\end{proof}

\vspace{0.2cm} In light of the results of the previous sections, the
proofs of Theorem 1.1 and 1.2 are basically an application of the Classification of
Finite Simple Groups (CFSG). The main point of this section is to prove Theorem 1.3,
which requires a bit more work.

\begin{proof}[Proof of Theorem \ref{simple3fp}]
Let $\Omega$ be a set such that $(G, \Omega)$ satisfies Hypothesis
\ref{3fix} and such that $G$ is simple. Then we apply the CFSG and
Theorems \ref{altmain}, \ref{liemain} and \ref{mainspor}. This gives
exactly the possibilities that are listed in Theorem
\ref{simple3fp}.
\end{proof}

\begin{proof}[Proof of Theorem \ref{almostsimple3fp}]

Let $\Omega$ be a set such that $(G, \Omega)$ satisfies Hypothesis
\ref{3fix} and suppose that $G$ is almost simple, but not simple.
Then Lemma \ref{hypcomp} implies that either $F^*(G) \cong
\PSL_2(2^p)$ with $p$ a prime, which is conclusion (1), or
$(F^*(G),\Omega)$ satisfies Hypothesis \ref{3fix}. If $F^*(G)$ is an
alternating group, then Theorem \ref{altmain} yields that $\Sym_5$
acting on $5$ points is the only example. But in light of the
isomorphism $\Sym_5 \cong \au (\PSL_2(4))$ we see that this example
is a special case of conclusion (1). If $F^*(G)$ is of Lie type,
then Lemmas \ref{AutPSL3}, \ref{AutPSU3} and \ref{autsingles} show
that (2) and (3) are the only possible examples.

Finally if $F^*(G)$ is sporadic, then $F^*(G)$ is isomorphic to
$M_{11}$ or to $M_{22}$. Our hypothesis that $G$ is not simple
implies that only the latter case can occur and in fact $G \cong
\au(M_{22})$. Let $\omega \in \Omega$. Then $G_\omega$ contains a
Sylow $7$-subgroup $S$ of $G$ and $N_{F^*(G)}(S) \cap G_\omega = S$.
Now $|N_G(S)/N_{F^*(G)}(S)| = 2$ and thus Lemma \ref{tuedel} forces
an involution $t$ into $G_\omega$. However $|C_{F^*(G)}(t)| = 1344$
and then Lemma \ref{tuedel} gives that $N_{F^*(G)}(S) \cap G_\omega
\neq S$, which is a contradiction.
\end{proof}

\begin{lemma}\label{3strich}
Suppose that Hypothesis \ref{3fix} holds. Then $3 \in \pi(G)$.
\end{lemma}

\begin{proof}
Assume otherwise, choose $G$ to be a minimal counter-example and let
$\alpha \in \Omega$. First we consider the case where $G_\alpha$ has
odd order. Let $1 \neq H \le G_\alpha$ be a three point stabilizer,
fixing the distinct points $\alpha, \beta$ and $\gamma$ of $\Omega$.
Let $1 \neq X \le H$ and $g \in N_G(X)$. As $o(g)$ is coprime to $3$
by assumption, the fixed points of $X$ cannot be interchanged by $g$
in a $3$-cycle. But the fact that point stabilizers have odd order
also implies that $g$ cannot interchange two of the points $\alpha,
\beta, \gamma$ and fix the third. Thus it fixes them all and is
hence contained in $H$. Now Lemma \ref{charfrob} forces $G$ to be a
Frobenius group, contrary to Hypothesis \ref{3fix}.

We conclude that $G_\alpha$ has even order.

\begin{enumerate}

\item[(1)]
If $G_\alpha$ contains a Sylow $2$-subgroup of $G$, then $G$ has
cyclic or quaternion Sylow $2$-subgroups.

\begin{proof}
Suppose that $G_\alpha$ contains a Sylow $2$-subgroup. Then
$O_2(G)=1$, and moreover $O_3(G)=1$ by assumption. If $E$ is a
component of $G$, then one of the cases from Lemma \ref{hypcomp}
holds. The first two cases are impossible by the assumption that $3
\notin \pi(G)$, and in the third case the main result of \cite{TW}
yields that $E/Z(E)$ is a Suzuki group. But this contradicts Lemma
\ref{Szq}. Hence $E(G)=1$ and $F^*(G)=F(G)$ is a $\{2,3\}'$-group.
Looking at Theorem \ref{minimalnormal}, we deduce that (a) holds and
therefore our claim follows.
\end{proof}

\item[(2)]
$G$ has a subgroup $M$ of index $2$.

\begin{proof}
First suppose that $G_\alpha$ contains a Sylow $2$-subgroup of $G$
and let $T$ be a $2$-subgroup of $G$. Then $T$ is cyclic or
quaternion by (1) and therefore $N_G(T)/C_G(T)$ is a $2$-group
(recall that $3 \notin \pi(G)$). So Frobenius' Theorem implies that
$G$ has a normal $2$-complement and hence a subgroup of index $2$.

Now two cases from Lemma \ref{syl2} remain, namely (2) and (3).
First suppose that $S \in \sy_2(G)$ is dihedral or semidihedral.
Then Frobenius' Theorem is applicable again and $G$ has a normal
$2$-complement, in particular a subgroup of index $2$.

Finally suppose that Lemma \ref{syl2}~(3) holds and let $\beta \in
\Omega$ be such that $S_\alpha=S_\beta$. Let $s \in S \setminus
S_\alpha$. We already treated the case where $S$ is cyclic, so we
may suppose that $o(s) \neq |S|$. Then $s$ induces a product of an
even number of cycles of $2$-power length on each regular $S$-orbit.
Moreover $s$ interchanges $\alpha$ and $\beta$ and therefore it
induces an odd permutation on $\Omega$. So again $G$ has a subgroup
of index $2$.
\end{proof}

\item[(3)]
Let $M$ be as in (2). Then $M$ acts transitively on $\Omega$.

\begin{proof}
Assume otherwise. Then $M$ has two orbits on $\Omega$ which we
denote by $\Delta_1$ and $\Delta_2$. Then the elements in $G
\setminus M$ interchange $\Delta_1$ and $\Delta_2$, so they have no
fixed points. By Hypothesis \ref{3fix} we find $y \in M_\alpha$ such
that $y$ fixes three points on $\Omega$. We may choose $y$ of prime
order $p$ and we may suppose that $\alpha \in \Delta_1$. If $\alpha$
is the unique fixed point of $y$ on $\Delta_1$, then $|\Delta_1|
\equiv 1$ modulo $p$ and it follows that $y$ also has a unique fixed
point on $\Delta_2$. But then $y$ cannot have three fixed points in
total, so this is impossible. With similar arguments it follows
that, if $y$ has two fixed points on $\Delta_1$, then it has two or
zero fixed points on $\Delta_2$, which again gives a contradiction.

Thus the only remaining possibility is that all fixed points of $y$
are contained in $\Delta_1$. In particular $|\Delta_1| \equiv 3$
modulo $p$. Then $y$ acts without fixed points on $\Delta_2$ and it
follows that $|\Delta_2| \equiv 0$ modulo $p$. As
$|\Delta_1|=|\Delta_2|$, this forces $p=3$, which is impossible.
This proves our claim that $M$ acts transitively on $\Omega$.
\end{proof}
\end{enumerate}

Let $M$ be as in (2) and (3). Since $3 \notin \pi(M)$ and $G$ is a
minimal counter-example, we know that $(M,\Omega)$ does not satisfy
Hypothesis \ref{3fix}. In particular the three point stabilizers in
$M$ are trivial, which forces $G \setminus M$ to contain elements
with three fixed points. As $|G:M|=2$, this implies that some
involution $t \in G$ fixes exactly three points and hence $|\Omega|$
is odd by Lemma \ref{sylow}~(a). Now (1) yields that $G$ has cyclic
or quaternion Sylow $2$-subgroups, and this forces $\langle t
\rangle \in \sy_2(G)$. In particular $M$ has odd order. It follows
with Lemmas \ref{charfrob} and \ref{tuedel} that $M$ acts regularly
on $\Omega$ or that $M$ is a Frobenius group. In the first case
$3=|\FO(t)|=|C_M(t)|$, contrary to the fact that $3 \notin \pi(M)$.
In the second case we let $K$ denote the Frobenius kernel of $M$.
Then $K$ acts regularly on $\Omega$ and $t$ normalizes it, so we
have the same contradiction as above.
\end{proof}

This already proves one of the statements in Theorem \ref{main}. For
the additional details, we split our analysis in two parts.

\begin{prop} \label{even}
Let $\Omega$ be a set such that $(G, \Omega)$ satisfies Hypothesis
\ref{3fix} and let $\omega \in \Omega$. If $|G_\omega|$ is even,
then one of the following is true:
\begin{enumerate}
\item[(1)]
$G$ has a normal $2$-complement.

\item[(2)]
$G$ has dihedral or semidihedral Sylow $2$-subgroups and $4$ does
not divide $|G_\omega|$. In particular $G_\omega$ has a normal
$2$-complement.

\item[(3)]
$G_\omega$ contains a Sylow $2$-subgroup $S$ of $G$ and $G$ has a
strongly embedded subgroup.

\item[(4)]
$|G:G_\omega|$ is even, but not divisible by $4$ and $G$ has a
subgroup of index $2$ that has a strongly embedded subgroup.
\end{enumerate}
\end{prop}

\begin{proof}
By hypothesis one of the cases (2), (3) or (4) from Lemma \ref{syl2}
holds. Case (2) leads to possibility (2) of our proposition. In Case
(4) we apply Lemma \ref{odd-even}, where one of the possibilities
(2), (3) or (4) holds. They lead to the cases (3), (1) and (4) of
our proposition. Finally we suppose that Lemma \ref{syl2}~(3) holds.
Then either $S$ is cyclic, which leads to (1), or some elements of
$S^\#$ act as odd permutations on $\Omega$ and hence $G$ has a
subgroup $G_0$ of index $2$. Let $S_0:=G_0 \cap S$. Then $S_0$ fixes
exactly two points $\alpha,\omega$ on $\Omega$. Let $M$ denote the
set-wise stabilizer of $\{\alpha,\omega\}$ in $G_0$.

Let $g \in G_0$ an let $1 \neq x \in M \cap M^g$ be a $2$-element,
without loss $x \in S_0$. Then $x$ fixes $\alpha$ and $\omega$ and
it is contained in a Sylow $2$-subgroup of $M^g$, hence without loss
it fixes $\alpha^g$ and $\omega^g$. Lemma \ref{sylow}~(a) implies
that $x$ does not have three fixed points, so
$\{\alpha,\omega\}=\{\alpha^g,\omega^g\}$ and therefore $g \in M$.
This shows that $M$ is a strongly embedded subgroup of $G_0$ as in
(4).
\end{proof}

\begin{prop} \label{oddZ3}
Let $\Omega$ be a set such that $(G, \Omega)$ satisfies Hypothesis
\ref{3fix} and let $\omega \in \Omega$. Suppose that $|G_\omega|$ is
odd. If $|\FO(G_\omega)| =3$, then one of the following is true:
\begin{enumerate}
\item[(1)] $G$ has a normal subgroup $R$ of order $27$ or $9$, and $G/R$ is
isomorphic to $\Sym_3$, $\Alt_4$, $\Sym_4$, to a fours group or to a
dihedral group of order $8$.
\item[(2)] $G$ has a regular normal subgroup.
\item[(3)] $G$ has a normal subgroup $F$ of index $3$ which acts as
a Frobenius group on its three orbits.
\item[(4)] $G$ has a normal subgroup $N$ which acts
semiregularly on $\Omega$ such that $G/N$ is almost simple and
$G_\omega$ is cyclic.
\end{enumerate}
\end{prop}

\begin{proof}
If $G_\omega$ is not t.i., then the main theorem of \cite{PS2}
implies that $G$ has a regular normal subgroup of order $27$ or $9$.
The structure of $G/R$ is given in the corollary to the main theorem
of \cite{PS2}.

On the other hand if $G_\omega$ is t.i. and $3$ is a divisor of
$|G_\omega|$, then Proposition 6.5 of \cite{PS1} implies that $G$
has a normal subgroup $N$ of index $3$. If the action of $N$ on
$\Omega$ is transitive, then by induction over the order of an
example we can see that $N$ contains a regular normal subgroup
$N_0$, or a normal index $3$ Frobenius group $F_0$. In the first
case a Frattini argument implies that $G = N_0G_\omega = G_\omega
N_0$ and thus $N_0$ is normal in $G$, proving that $G$ possesses a
regular normal subgroup. In the second case the Frobenius kernel
$K_0$ of $F_0$ is a characteristic subgroup of $F_0$, and is hence
also normal in $G$. The number of $F_0$-orbits on $\Omega$ is equal
to $3$, thus the orbit stabilizer $G_0$ in $G$ of one of the
$F_0$-orbits acts as a Frobenius group on its fixed orbit, and hence
on all $F_0$ orbits; i.e. $G_0$ is a Frobenius group of index $3$ in
$G$.  As $G/F_0$ has order $9$, every index three subgroup of
$G/F_0$ is normal. Thus $G_0 \unlhd G$, which is one of our possible
conclusions.

Finally we consider the case where $G_\omega$ is still a t.i.
subgroup and moreover $|G_\omega|$ is coprime to $6$. If $G$ is
solvable, then Proposition 3.1 of \cite{PS1} shows that either (2)
or (3) holds. Thus we may assume that $G$ is not solvable. If $N$ is
a minimal normal subgroup of $G$ and $N$ is abelian of order $r^k$,
then $N \cap G_\omega =1$ either by Lemma \ref{sylow} if $r \neq
2,3$, or because we are assuming that $|G_\omega|$ is coprime to
$6$. In every case $N$ must act semiregularly on $\Omega$. If $r
\neq 3$, then Lemma 1.9 in \cite{PS1} implies that $G_\omega$ has at
most one fixed point on $\omega^N$. If $r = 3$, then $N$ is an
elementary abelian $3$-group and thus so is $C_N(G_\omega)$.

As $|C_N(G_\omega)| = |{\textrm{F}_{\omega^N}}|$ it must be that either
$\FO(G_\omega) \cap \omega^N = \{\omega \}$,
which is what we want, or that $|\FO(G_\omega) \cap \omega^N | =3$
and thus $|C_N(G_\omega)| = 3$.

If $|C_N(G_\omega)| = 3$, then Lemma \ref{o3} implies that
$G_\omega$ has a normal complement $K$ in $G$. As $|K||G_\Omega| =
|G| = |\Omega||G_\omega|$, we obtain that $K \cap G_\omega = 1$ and
thus that $K$ is a regular normal subgroup. This is one of our
conclusions.

So if $H$ does not posses a normal complement in $G$, then every
abelian minimal normal subgroup $N$ of $G$ acts semiregularly on
$\Omega$ and $\FO(G_\omega)$ intersects an $N$-orbit in at most one
point. If $r \neq 3$ and conclusion (3) does not hold, then Lemma
1.9 of \cite{PS1} asserts that the action of $G/N$ on $\ti{\Omega}$,
the set of $N$-orbits on $\Omega$, is faithful. Moreover
$G_{\ti{\omega}} \cong G_\omega$ and every $x \in G_{\ti{\omega}}$
fixes either $3$ or no points of $\ti{\Omega}$. So  $(G/N,
\ti{\Omega})$ is a $(0,3)$ group and in particular satisfies
Hypothesis \ref{3fix}.

If $r = 3$, then we saw that $C_N(x) = 1$ for all $1 \neq x \in G_\omega$. Thus $x$ fixes exactly
$3$ orbits of $N$. On each of these the action of $NG_\omega$ is Frobenius. If $|\ti{\Omega}| = 3$,
then $NG_\omega$ is an index three Frobenius subgroup of $G$ and the action of $G$ on
$\ti{\Omega}$ is either cyclic or $\Sym_3$. The latter case can not happen
as $G_\omega$ is odd. So $NG_\omega$ is the kernel of the action of $G$ on $\ti{\Omega}$ and
hence is normal in $G$. So if $|\ti{\Omega}| = 3$, then conclusion (3) holds.
Thus we may assume that $|\ti{\Omega}| > 3$. The kernel of the action of $G$ on $\ti{\Omega}$
lies in the stabilizer of $\omega^N$ which is $NG_\omega$. As $NG_\omega$ is a Frobenius group with complement
$G_\omega$ the kernel of the action must lie inside $N$, which implies that $G/N$ acts
faithfully on $\ti{\Omega}$. Also $G_{\ti{\omega}} \cong G_\omega$ and every $x \in
G/N$ fixes either $3$ or no points of $\ti{\Omega}$.

Thus if conclusion (3) does not hold, then $(G/N, \ti{\Omega})$ satisfies Hypothesis \ref{3fix} and
that moreover if $x \in G/N$ then $x$ fixes either $3$ or no points of $\ti{\Omega}$.

Thus if conclusions (2) and (3) do not hold for $G$ and $N$ has an
abelian minimal normal subgroup, then by induction on $|G|$ we may
conclude that (4) holds for $(G/N,\ti{\Omega})$. In turn this
implies that conclusion (4) holds for $G$. On the other hand if
conclusions (2) and (3) do not hold for $G$ and $N$ does not have an
abelian minimal normal subgroup, then by Theorem \ref{minimalnormal}
we see that $G$ is almost simple and the action on $\Omega$ must
satisfy Hypothesis \ref{3fix}. (The case with $F^*(G) = \PSL_2(2^p)$
implies that $|G_\omega|$ is even, hence it is not allowed here.)
Inspection of the simple and almost simple examples now yields that
$G_\omega$ is cyclic. Thus again conclusion (4) holds and our proof
is complete.
\end{proof}

\begin{proof}[Proof of Theorem \ref{main}]
Let $\Omega$ be a set such that $(G, \Omega)$ satisfies Hypothesis
\ref{3fix}. Then $G$ has order divisible by $3$ by Lemma
\ref{3strich}. If $\omega \in \Omega$, then we first consider the
case where $G_\omega$ has even order. Then Lemma \ref{even} gives
exactly the possibilities in Theorem \ref{main}.1. Next we suppose
that $G_\omega$ has odd order. Then Corollary \ref{FrobTo3} reduces
our situation to the case of $(0,3)$ groups, so Proposition
\ref{oddZ3} is applicable. It yields the details in Theorem
\ref{main}.2.
\end{proof}

We now consider the situation where $G_\omega$ is a Frobenius group
of odd order. We note that Corollary \ref{FrobTo3} implies that the
action of $G$ on the set of cosets of a nontrivial three point stabilizer $H$ is of type
$(0,3)$ and thus one of the conclusions of Proposition \ref{oddZ3}
holds. Conclusion (4) is impossible because $G_\omega$ is a
Frobenius group by hypothesis, so in particular it is not cyclic.
Conclusions (1) and (3) pin down the structure of $G$ as best as
possible and thus we now consider the situation where $G$ has a
regular normal subgroup.

\begin{cor}
Let $\Omega$ be a set such that $(G, \Omega)$ satisfies Hypothesis
\ref{3fix} and let $\omega \in \Omega$. Suppose further that
$|G_\omega|$ is odd. If $|\FO(G_\omega)| \neq 3$ and $G$ contains a
regular normal subgroup $N$ in its action on the set of cosets of
the three point stabilizer $H$, then $N$ is a Frobenius group with
complement $N_\omega$ and $\FO(x) = \{\omega \}$ for all $x \in
N_\omega^\#.$
\end{cor}

\begin{proof}
As $|\FO(G_\omega)| \neq 3$ Lemma \ref{cases} implies that
$G_\omega$ is a Frobenius group with kernel $N_\omega$ and
complement $H$. We will show that $N_N(X) = N_\omega$ for all $1
\neq X \leq N_\omega$ which, by Lemma \ref{charfrob}, implies our
claim. Now by Hypothesis $G = NH$ and thus $N_\omega$ fixes at most
two points of $\Omega$. Thus Lemma \ref{tuedel} implies that
$|N_N(N_\omega):N_\omega| \leq 2.$ Suppose now that
$|N_N(X):N_\omega| = 2$ for some $1 \neq X \leq N_\omega$. Then
$|N_N(N_\omega):N_\omega|=2$ and $N_G(N_\omega) = N_N(N_\omega)H$.
Thus by a Frattini argument some involution $t \in N \setminus
N_\omega$ normalizes $H$. But by hypothesis $|N_G(H):H| = 3$, so $t$
must be $G$-conjugate to an element of $H$. This is impossible as
$|H|$ is odd. Therefore $N_N(X) = N_\omega$ for all $1 \neq X \leq
N_\omega$, establishing that $N$ is a Frobenius group. We also see
that every nontrivial element of $N_\omega$ fixes a unique point.
\end{proof}

\begin{cor} \label{oddFrobconclusion}
Let $\Omega$ be a set such that $(G, \Omega)$ satisfies Hypothesis
\ref{3fix} and that $|G_\omega|$ is odd. If $|\FO(G_\omega)| \neq
3$, and $G$ contains a regular normal subgroup $N$ in its action on
the cosets of the three point stabilizer $H$, then $G$ is one of the
groups from Lemma \ref{BeispielKoerper}.
\end{cor}
\begin{proof}
The previous corollary established that $G$ is solvable. Next we
observe that $G_\omega$ has exactly three orbits which are not
regular. To see this let $\FO(H) = \{\omega, \omega_1, \omega_2\}$
and consider the $N_\omega$-orbits of $\omega_1$ and $\omega_2$.
These are regular $N_\omega$-orbits on which $H$ has exactly one
fixed point. Thus $G$ satisfies the hypotheses of Lemma 4.3 in
\cite{PS2}. The conclusion of Lemma 4.3 is that $G$ contains a
normal subgroup $A$ which is isomorphic to the additive group of a
finite field of order $3^p$ where $p$ is prime. $N_\omega$ is a
subgroup of the multiplicative group of this field whereas $H$ is
the Galois group of the field. These are precisely the examples in
Lemma \ref{BeispielKoerper}.
\end{proof}

The final two corollaries give additional information for conclusion
(2) of Proposition \ref{oddZ3}.

\begin{cor}\label{oddregnot3conclusion}
Let $\Omega$ be a set such that $(G, \Omega)$ satisfies Hypothesis
\ref{3fix}. Let $\omega \in \Omega$ and suppose that $|G_\omega|$ is
odd. If $|\FO(G_\omega)| = 3$ and $G$ contains a regular normal
subgroup $N$, and if moreover $G_\omega$ is not a $3$-group, then
$N$ is solvable and $N = O_{3,3'}(N)C_N(x)$ for some $x \in
G_\omega$ with $|C_N(x)| = 3$.
\end{cor}

\begin{proof}
As $G_\omega$ is not a $3$-group, there exists $x \in G_\omega$ of
prime order $p > 3$. Now Lemma \ref{sylow} implies that $o(x)$ and
$|N|$ are coprime. Thus the hypotheses of Lemma \ref{Fukushima} are
satisfied, which implies our conclusion.
\end{proof}

\begin{cor}\label{oddregHeq3conclusion}
Suppose that Hypothesis \ref{3fix} holds and let $\omega \in
\Omega$. Suppose that $|G_\omega|$ is odd. If $|\FO(G_\omega)| = 3$
and $G$ contains a regular normal subgroup $N$, and if moreover
$G_\omega$ is a $3$-group, then $G_\omega$ is cyclic and $N$ is
solvable such that $G = O_{3,3'}(N)G_\omega$.
\end{cor}

\begin{proof}
We let $H:=G_\omega$ and consider $P \in \sy_3(G)$ such that $H \leq
P$. Now Lemma \ref{sylow} implies that either
$|P:H| \leq 3$, or that $|H| = 3$ and $P$ has maximal class. %

First we consider the case $|P:H| \leq 3$, i.e. $|P \cap N| \leq 3$.

If $|P \cap N| = 1$, then $|N|$ is a $3'$-group whereas $|C_N(H)| =
3$; a contradiction.

If $|P \cap N| = 3$, then either $C_N(P \cap N) = N_N(P \cap N)$ and
$N$ has a normal $3$-complement $K$, or $H$ centralizes $N_N(P \cap
N) / C_N(P \cap N)$ (which is of order $2$). In the latter case
$C_H(N)$ has even order, contrary to our hypothesis. In the former
case $G = KP$ and, as $|N:K| = 3$ and $N$ is regular on $\Omega$, we
obtain that $K$ has three orbits on $\Omega$ and that $KH$ is a
Frobenius group. Thus $H$ is cyclic and $N = O_{3,3'}(N)$, which is
our conclusion.

So now we consider the situation where $|H| = 3$ and $P$ has maximal
class. We see that Lemma 1.9 of \cite{PS1} implies that $O_{3'}(N)$
acts semiregularly on $\Omega$ and that the action of $G/O_{3'}(N)$
is faithful on the set $\ti{\Omega}$ of $O_{3'}(N)$-orbits. Now,
since no almost simple group can satisfy Hypothesis \ref{3fix} with
$H$ a $3$-group (see Theorem \ref{almostsimple3fp}), we see that
$F^*(G/O_{3'}(N)) = F(G/O_{3'}(N)) = O_{3}(G)(G/O_{3'}(G))$ and
$O_{3}(G)(G/O_{3'}(G))$ acts semiregularly on $\ti{\Omega}$. Thus
one of the following could happen:

$O_{3}(G)(G/O_{3'}(G))$ could act regularly on $ \ti{\Omega}$, or it
could act semiregularly with at least three orbits on $\ti{\Omega}$.

However the latter possibility does not occur because $H$ fixes
three points on any $H$-invariant $O_{3}(G)(G/O_{3'}(G))$-orbit,
since a $3$-group never acts fixed point freely on a $3$-group. Thus
$N/O_{3'}(G)$ acts regularly on $\ti{\Omega}$, which implies that
$G/O_{3'}(G)$ is a $3$-group because $|G/O_{3'}(G)| =
|\ti{\Omega}|\cdot |H| = |N/O_{3'}(G)|\cdot |H|$. Again our claim
follows.
\end{proof}







\end{document}